\documentclass[a4paper]{amsart}

\usepackage{amsmath,amssymb,amsfonts,mathtools}
\usepackage{bm}
\usepackage{mathrsfs}
\usepackage{enumitem}
\usepackage{hyperref}
\usepackage{microtype}
\usepackage{booktabs}
\usepackage{cite}

\usepackage{mathtools}
\mathtoolsset{showonlyrefs}

\usepackage[
    a4paper,
    margin=0.7in
]{geometry}

\hypersetup{
  colorlinks=true,
  linkcolor=blue,
  citecolor=blue,
  urlcolor=blue
}

\allowdisplaybreaks

\newtheorem{theorem}{Theorem}[section]
\newtheorem{proposition}[theorem]{Proposition}
\newtheorem{lemma}[theorem]{Lemma}
\newtheorem{corollary}[theorem]{Corollary}

\theoremstyle{definition}
\newtheorem{definition}[theorem]{Definition}
\theoremstyle{plain}
\newtheorem{remark}[theorem]{Remark}

\newcommand{\R}{\mathbb{R}}
\newcommand{\C}{\mathbb{C}}

\newcommand{\Hh}{\mathcal{H}}
\newcommand{\eps}{\varepsilon}
\newcommand{\dd}{\,\mathrm{d}}

\newcommand{\sgn}{\operatorname{sgn}}
\newcommand{\Rea}{\operatorname{Re}}

\newcommand{\norm}[1]{\left\|#1\right\|}

\newcommand{\ip}[2]{\left\langle #1,#2\right\rangle}

\title[modified Benjamin--Ono equation]{Local well-posedness for the exponentially modified Benjamin-Ono equation}

\author[T. S. R. Santos]{Thyago S. R. Santos}
\address{Departamento de Matemática, Instituto de Matemática, Estatística e Computação Científica,
Universidade Estadual de Campinas, Campinas, SP 13083-859, Brazil}
\email{thyagosr@unicamp.br}


\begin{document}

\begin{abstract}

We prove local well-posedness for the modified Benjamin-Ono equation in $H^s(\R)$ for $s>3/2$. This equation arises as a modified counterpart of the Benjamin-Ono equation through a Miura-type transformation, in close analogy with the classical relation between mKdV and KdV equations. One of the main obstacles is the apparent derivative loss in the quadratic part of the equation. To overcome this difficulty, we introduce a Cole-Hopf transform adapted to the structure of the equation. The map removes the quadratic derivative term and converts the equation into a Schrödinger equation with a transport and a commutator terms that can be controlled without derivative loss. Solutions are then constructed by parabolic regularization, while frequency envelopes are used to obtain the continuity of the data-to-solution map. In the degenerate case $\gamma=0$, the transform completely linearizes the equation and yields local well-posedness for every $s>1/2$, down to the scaling-critical threshold.
\end{abstract}

\maketitle

\section{Introduction}

\subsection{Setting and main results}

The Korteweg--de Vries (KdV) equation
$$
u_t+u_{xxx}+6uu_x=0
$$
is one of the classical models for the propagation of long, weakly nonlinear dispersive waves in shallow water. Its Cauchy problem has been extensively studied by several methods; see, for instance, \cite{Kato1983,KenigPonceVega1993,Bourgain1993KdV,KenigPonceVega1996,KillipVisan2019} and the references therein. A closely related model is the modified Korteweg--de Vries (mKdV) equation
$$
v_t+v_{xxx}+6v^2v_x=0,
$$
whose Cauchy theory has been developed in several works; see \cite{Kato1983,KenigPonceVega1993,KenigPonceVega2001,CollianderKeelStaffilaniTakaokaTao2003,HarropGriffithsKillipVisan2024}. The relation between KdV and mKdV is given by the Miura transformation. Indeed, if $v$ is a real-valued solution of the  mKdV equation, then
$$
\mathcal{M}_{\rm KdV}(v):=v^2+i v_x
$$
is a complex-valued solution of the KdV equation. Equivalent real-valued versions arise under the corresponding changes of sign convention. This connection originates in Miura's works \cite{Miura1968, Miura1976, MiuraGardnerKruskal1968} and remains one of the main structural links between KdV and mKdV.

A similar correspondence arises in the Benjamin--Ono (BO) equation
\begin{equation}\label{eq: BO}
u_t+\Hh u_{xx}+ 2uu_x=0,   \tag{BO}  
\end{equation}
where $\Hh$ denotes the Hilbert transform, defined on the Fourier side by $\widehat{\Hh f}(\xi)=-i\sgn(\xi)\widehat f(\xi)$. The BO equation arises as a model for the propagation of long internal waves in stratified fluids of great depth, and its Cauchy problem has been extensively studied; see, for example, \cite{Saut79, MR1044731, Iorio1986,Tao2004,IonescuKenig2007,MolinetPilod2012,KillipLaurensVisan2024}. As in the KdV setting, there is a modified equation
connected to BO through a Miura-type transformation. This equation takes the form
$$
u_t+\Hh u_{xx}+u_x\Hh u_x+\gamma(e^u-1)u_x=0.
$$
Indeed, for each $\gamma\in\R$ and every real-valued solution $u$, the Miura-type map
\begin{equation}\label{Miura Transform BO}
\mathcal{M}_{\rm BO}^{\gamma}(u)
:=\frac{1}{2}\left(
\Hh u_x+\gamma(e^u-1)+iu_x
\right)
\end{equation}
produces a complex-valued solution of the Benjamin--Ono equation. The connection between the Benjamin--Ono equation and its modified counterparts through Miura--B\"acklund transformations goes back to works from the late 1970s, in particular those of Bock and Kruskal \cite{BockKruskal1979} and Nakamura \cite{Nakamura1979}. This provides a nonlocal analogue of the classical KdV--mKdV correspondence and naturally explains the particular combination of the terms $\Hh u_{xx}$, $u_x\Hh u_x$, and $(e^u-1)u_x$; see also \cite[Section~3.6.6]{MR4400881}.

All these equations are closely related through completely integrable structures. For KdV and mKdV, this includes Lax pairs, inverse scattering, infinitely many conservation laws, and the classical Miura transformation; see, for instance, \cite{Lax1968,Miura1968,MiuraGardnerKruskal1968}. The BO equation also admits an inverse scattering formulation and an integrable hierarchy; see \cite{FokasAblowitz1983,FokasFuchssteiner1981,KaupMatsuno1998}. This framework naturally motivates the study of modified models and their connections with the original equations.

Motivated by the above discussion, in this work, we consider the following initial value problem
\begin{equation}\label{eq:embo}
\left\{
\begin{aligned}
&u_t+\Hh u_{xx}+u_x\Hh u_x+\gamma(e^u-1)u_x=0,\\
&u(0,x)=u_0(x), 
\end{aligned}
\right.\tag{e-mBO}
\end{equation}
where $\gamma\in\R$ is a fixed parameter. We emphasize that \eqref{eq:embo} should not be confused with the equation usually referred to as the cubic modified Benjamin--Ono equation in the literature,
\begin{equation}\label{cubic mBO}
 u_t+\Hh u_{xx}+u^2u_x=0,   
\end{equation}
whose local well-posedness theory has been developed in several works; see, for example, \cite{MolinetRibaud2004,KenigTakaoka2006,Guo2011mBO} and references therein. To avoid this ambiguity, we refer to \eqref{eq:embo} as the \emph{exponential modified Benjamin--Ono equation}, or simply the e-mBO equation. This terminology is motivated by the exponential nonlinearity
\begin{equation}\label{expansion}
(e^u-1)u_x:= uu_x + \frac{1}{2}u^2u_x + \frac{1}{6}u^3u_x + \frac{1}{24}u^4u_x+ \cdots=\sum_{n=1}^{\infty}\frac{u^n}{n!}u_x.  
\end{equation}
In particular, the equation differs substantially from the \eqref{cubic mBO}, both through this infinite analytic nonlinearity and through the quadratic derivative term $u_x\mathcal H u_x$, with one derivative falling on each factor.

As mentioned above, the equation in \eqref{eq:embo} has a previous history in the theory of integrable systems. Nakamura \cite{Nakamura1979} obtained explicit periodic and multisoliton solutions using the B\"acklund structure of the Benjamin--Ono equation. Later, Scoufis and Cosgrove \cite{ScoufisCosgrove1995} studied the initial-value problem by inverse-scattering methods, using the Fokas--Ablowitz transform for a complex extension of the Benjamin--Ono equation. More recently, Guo \cite{Guo2024} developed a B\"acklund transformation, a Lax representation, conservation laws, an inverse-scattering formulation, multisoliton solutions, and long-time asymptotics. These works approach the equation through its integrable structure. The purpose of the present paper is different: we develop a well-posedness theory in Sobolev spaces using PDE methods, without relying on inverse scattering.

The scaling of the equation provides a first indication of its natural regularity. When $\gamma=0$, equation \eqref{eq:embo} is invariant under the scaling
$$
u_\lambda(t,x)=u(\lambda^2t, \lambda x),
$$
thus the corresponding Sobolev critical index is $s_c:=1/2$. For $\gamma\ne0$, the exponential term $(e^u-1)u_x$ breaks this scaling. From the point of view of derivatives, however, it is of lower order than the quadratic term $u_x\Hh u_x$. Nevertheless, this lower-order contribution cannot be regarded as harmless, since its leading quadratic part $ uu_x$ has a decisive effect on the regularity of the flow map.

Indeed, for $t\in \R$, let $S(t):=e^{it\mathcal{H}\partial_x^2}$ be the linear Benjamin--Ono propagator, and define
\begin{equation}\label{operator B}
 \mathcal B(t)(\varphi,\psi)
:=
\int_0^t
S(t-t')B\bigl(S(t')\varphi,S(t')\psi\bigr)\dd t',   
\end{equation}
where
$$
B(f,g)
:=
\frac12\bigl(f_x\Hh g_x+g_x\Hh f_x\bigr)
+
\frac{\gamma}{2}\bigl(fg_x+gf_x\bigr).
$$
When $\gamma\ne0$, the term $uu_x$ in \eqref{expansion}, produces the same high--low frequency interaction responsible for the classical non-smoothness phenomena in Benjamin--Ono-type equations \cite{MolinetSautTzvetkov2001,KochTzvetkov2005}.
\begin{theorem}\label{thm:c2}
Let $\gamma\ne0$, $s\in\R$, and $t\ne0$. Then
$$
\mathcal B(t):H^s\times H^s\longrightarrow H^s
$$
is not a  bounded operator. Consequently, for every $T>0$, the Duhamel map associated with \eqref{eq:embo} cannot be $C^2$ as a map into $C([-T,T];H^s)$.
\end{theorem}

The standard second-iterate argument gives the obstruction. For completeness, we include the proof in the Appendix \ref{sec:c2}.

Thus, a standard Picard iteration based on the Duhamel formulation in \(H^s\) cannot provide the local theory, even at arbitrarily high regularity. Nevertheless, this obstruction does not prevent the existence of a continuous Sobolev flow. One of the main difficulties is the quadratic term
$$
u_x\Hh u_x.
$$
 At the level of a direct Sobolev energy estimate, this appears to lose one derivative: \(u\in H^s\) only gives \(u_x\mathcal H u_x\in H^{s-1}\), whereas an \(H^s\) estimate for the equation would require control at the same differentiated level.

The equation, however, has a cancellation that removes this apparent derivative loss. If $P_+$ and $P_-$ denote the projections onto the positive and negative frequencies, respectively, then we have
$$
u_x\Hh u_x
=
i(P_-u_x)^2-i(P_+u_x)^2,
$$
so interactions between positive and negative frequencies cancel. After projecting onto positive frequencies, we introduce the nonlinear change of variables
$$
e^{P_+u}-1,
$$
which exactly absorbs the remaining quadratic derivative term into the Schrödinger operator. The remaining differentiated term appears through
$$
[P_+,e^u-1]u_x,
$$
which can be estimated without losing derivatives. This allows us to close the estimates in $H^s$ for $s>3/2$.

This structure leads to the following local well-posedness result.

\begin{theorem}\label{thm:lwp}
Let $\gamma\in\R$ and $s>3/2$. For every $R>0$ there exists $T=T(s,|\gamma|,R)>0$ such that, for every $u_0\in H^s(\R)$ with $\norm{u_0}_{H^s}\le R$, there exists a unique real-valued solution of \eqref{eq:embo}
$$
u\in C([-T,T];H^s)
\cap
C^1([-T,T];H^{s-2}).
$$
Moreover, the data-to-solution map
$$
u_0\longmapsto u
$$
is continuous from $\left\{u_0\in H^s(\mathbb R):
\|u_0\|_{H^s}\le R\right\}$ to $C([-T,T];H^s)$. If, in addition, $u_0\in H^\sigma(\R)$ for some $\sigma>s$, then $u\in C([-T,T];H^\sigma)$.
\end{theorem}

As a standard consequence of the local well-posedness theory and the usual continuation argument, we obtain the following blow-up alternative. Since the argument is classical, we omit the proof.

\begin{corollary}\label{Blow up alternative}
Let $\gamma\in\R$, $s>3/2$, and $u_0\in H^s(\R)$. Let
$$
u\in C((T_-,T_+);H^s)
$$
be the maximal solution to \eqref{eq:embo} with $u(0)=u_0$. If $T_+<\infty$, then
$$
\limsup_{t\nearrow T_{+}}\norm{u(t)}_{H^s}=+\infty.
$$
An analogous result holds if $T_->-\infty$.
\end{corollary}

When $\gamma=0$, we obtain the sharp local well-posedness result.

\begin{theorem}\label{corolario s>sc}
Let $\gamma=0$. Then the Cauchy problem \eqref{eq:embo} is locally well-posed in $H^s(\R)$ for every $s>s_c$.
\end{theorem}

\begin{remark}
As in the classical KdV--mKdV correspondence, the Miura-type transformation \eqref{Miura Transform BO} maps real-valued solutions of \eqref{eq:embo} into complex-valued solutions of the BO equation \eqref{eq: BO}; see \cite{ScoufisCosgrove1995, MR4400881}. For $s>1/2$, it maps $H^s(\R\to\R)$ into $H^{s-1}(\R\to\C)$. Thus, Theorem \ref{thm:lwp} produces BO solutions in $H^{s-1}(\R\to\C)$ for $s>3/2$. In the case $\gamma=0$, Theorem \ref{corolario s>sc} gives solutions for every $s>1/2$, and hence the Miura transformation yields complex-valued BO solutions in $H^\sigma(\R\to\C)$ for every $\sigma>-1/2$, matching the sharp Sobolev threshold for the BO equation; see \cite{KillipLaurensVisan2024}.
\end{remark}


\subsection{Outline of the method}

We briefly describe the main steps in the proof of Theorem \ref{thm:lwp}. First, we introduce the Cole--Hopf type change of variables
$$
U:=e^{P_+u}-1.
$$
The positive-frequency component is then transformed into
$$
U_t-iU_{xx}+\gamma (e^u-1)U_x
=
-\gamma(1+U)[P_+,e^u-1]u_x.
$$
This eliminates the apparent derivative loss and leaves a nonlinear Schr\"odinger equation with a commutator structure in the remaining nonlinear term.

Second, the frequency structure of the commutator
$$
[P_+,e^u-1]u_x
$$
forces the frequency of $e^u-1$ to be at least comparable to that of the differentiated factor. This allows the derivative to be controlled by the coefficient and gives, for $ s>3/2$,
$$
\norm{[P_+,e^u-1]u_x}_{H^s}
\lesssim
\norm{e^u-1}_{H^s}\norm{u}_{H^s}.
$$
Together with the standard estimates for $e^u-1$, this yields an $H^s$ bound for the transformed equation depending only on the $\|u\|_{H^s}$. The transformation can then be inverted through
$$
u=2\ln|1+U|.
$$

Third, we construct solutions by means of a \textit{parabolic regularization} method. We add a artificial dissipative term, solve the regularized equation for smooth data, and derive estimates that are uniform in the regularization parameter. These bounds provide a common existence time depending only on the size of the initial data in $H^s$. We also keep $|1+U|$ away from zero so that the inverse transformation remains well defined. Compactness then gives a solution of the original equation.

The remaining part of the proof is devoted to uniqueness and continuous dependence. For two solutions, a direct estimate of their difference in $H^s$ would again place a derivative on the difference. We therefore first estimate the difference in a weaker Sobolev norm, which yields uniqueness and stability at lower regularity.

To recover the full topology, we use frequency envelopes. If $(c_k)_{k \geq 0}$ describes the dyadic size of the initial data, the energy estimates show that the same dyadic decay is propagated by the solution, \textit{i.e.}
$$
\sup_{|t|\le T}
2^{sk}\norm{P_k u(t)}_{L^2}
\lesssim c_k.
$$
In particular, high-frequency tails remain uniformly small when the corresponding tails of the initial data are small.

We then split the difference of two solutions into low and high frequencies. The low-frequency part is controlled by the weaker stability estimate, while the high-frequency parts are controlled by the propagated envelopes. This gives continuity of the data-to-solution map in $H^s$.

Finally, the same argument propagates stronger frequency decay. Hence, if $u_0\in H^\sigma$ with $\sigma>s$, the solution remains in $H^\sigma$ on the same time interval. The frequency-envelope bounds also give strong continuity in time in the corresponding Sobolev spaces.

\begin{remark}
    The case $\gamma=0$ is considerably simpler. In fact, the transformed variable $U=e^{P_+u}-1$ satisfies the free Schrödinger equation, and the solution can be recovered locally by preserving the positive-frequency structure and keeping $1+U$ away from zero (see Section \ref{rem:gamma0}).
\end{remark}

\begin{remark}
To the best of our knowledge, this is the first Sobolev well-posedness theory for the e-mBO equation obtained by PDE methods, without using inverse-scattering or other integrability-based techniques.
\end{remark}

\begin{remark}
The condition $s>3/2$ in Theorem \ref{thm:lwp} arises from the energy estimates. It would be interesting to investigate whether the regularity threshold can be lowered by combining the ideas developed here with more refined energy estimates; see, for instance, \cite{MR3397003, MR3906854, MolinetPilodVento2025, MR5027534}.
\end{remark}

\subsection*{Organization of the paper}

Section \ref{sec:notation} fixes notation and records some standard estimates. Section \ref{sec:structure} analyzes the frequency decomposition, the quadratic resonance, and the Cole--Hopf type transform. Section \ref{sec:commutator} establishes the commutator, composition, and transformed energy estimates.  Sections \ref{sec:existence} and \ref{sec:stability} construct the solution and prove uniqueness, continuous dependence and persistence of solutions. Finally, in the Section \ref{rem:gamma0} proves Theorem \ref{corolario s>sc}.

\subsection*{Acknowledgments}
The author was partially supported  by São Paulo Research Foundation (FAPESP) grant No.~2024/15587-1. The author would also like to thank Professor Felipe Linares for reading an earlier version of this manuscript and for his valuable comments and suggestions.

\section{Notations and preliminary lemmas}\label{sec:notation}

We write $A\lesssim B$ if $A\le CB$ for some constant $C>0$ depending only on fixed parameters, and $A\sim B$ if both $A\lesssim B$ and $B\lesssim A$. Sometimes, when the constant depends on a parameter, we indicate this by a subscript, for example $A\lesssim_s B$.

For $s\in\R$, we define the Bessel potentials of order $-s$ by \footnote{Along this work, we use the Fourier transform convention $\widehat f(\xi)=\mathcal{F}[f](\xi):=\int_\R e^{-i x\xi}f(x)\dd x.$}
$$
\widehat{J^sf}(\xi):= \langle\xi\rangle^s \widehat{f}(\xi)=\left(1+|\xi|^2\right)^{s/2} \widehat{f}(\xi),
$$
 We will consider $H^s:=H^s(\R)$, with $s \in \R$, the usual inhomogeneous Sobolev space, equipped with the norm
$$
\|f\|_{H^s}: =\|J^s f \|_{L^2}.
$$

For $t\in \R$, we denote by $S(t)=e^{it\Hh\partial_x^2}$ the free group associated with the linearized Benjamin--Ono equation, i.e.
$$
\widehat{S(t)f}(\xi)
:=
e^{-i t\xi|\xi|}\widehat{f}(\xi).
$$

Let $P_+$ and $P_-$ denote the projections onto the positive and negative Fourier frequencies, respectively. They are given by
$$
P_{\pm}f=\frac12(f\pm i\Hh f)=\left(\mathbf 1_{\R_\pm}\widehat{f}\right)^\vee.
$$
The basic identities are
$$
P_++P_-=I,
\quad
P_\pm^2=P_\pm,
\quad
P_+P_-=P_-P_+ \equiv 0,
$$
and
$$
\Hh P_\pm= \mp i P_\pm.
$$
For real-valued $f$, one also has $P_-f=\overline{P_+f}$.

\begin{proposition}\label{prop:null}
For every sufficiently regular $f$,
$$
f_x\Hh f_x
=-i(P_+f_x)^2+i(P_-f_x)^2.
$$
\end{proposition}

\begin{proof}
Since we can write
$$
f_x=P_+f_x+P_-f_x,
$$
and
$$
\Hh f_x=-i P_+f_x +i P_-f_x,
$$
we have
$$
f_x\Hh f_x
=(P_+f_x +P_-f_x)(-i P_+f_x +iP_-f_x)
=-i(P_+f_x)^2+i(P_-f_x)^2.
$$
\end{proof}

The following lemma collects the basic composition and tame estimates for the exponential nonlinearity that will be used throughout the paper.

\begin{lemma}\label{lem:exp}
Let $s>1/2$. There is $C_s\ge1$ such that, for $f,g\in H^s$
\begin{align}
\norm{e^f-1}_{H^s}
&\le \norm f_{H^s}e^{C_s\norm f_{H^s}},\label{eq:exp1}\\
\norm{e^f-f-1}_{H^s}
&\le C_s\norm f_{H^s}^2e^{C_s\norm f_{H^s}},\label{eq:exp2}\\
\norm{e^f-e^g}_{H^s}
&\le C_s e^{C_s(\norm f_{H^s}+\norm g_{H^s})}\norm{f-g}_{H^s}.
\label{eq:exp3}
\end{align}
If $\sigma>s>1/2$, then on bounded subsets of $H^s\cap H^\sigma$ one has
$$
\norm{e^f-1}_{H^\sigma}
\le C_{s,\sigma}(\norm f_{H^s})\norm f_{H^\sigma}.
$$
\end{lemma}

\begin{proof}
Since $s>1/2$, we have
\begin{equation}\label{eq:Hs-algebra-exp}
\norm{fg}_{H^s}
\le C_s\norm f_{H^s}\norm g_{H^s}.
\end{equation}
Iterating \eqref{eq:Hs-algebra-exp}, for every integer $n\ge1$,
$$
\norm{f^n}_{H^s}
\le
C_s^{n-1}\norm f_{H^s}^n.
$$
Therefore, using
$$
e^f-1=\sum_{n=1}^\infty\frac{f^n}{n!},
$$
we obtain
\begin{align*}
\norm{e^f-1}_{H^s}
&\le
\sum_{n=1}^\infty\frac{1}{n!}\norm{f^n}_{H^s}\le
\sum_{n=1}^\infty
\frac{C_s^{n-1}}{n!}\norm f_{H^s}^n=
\norm f_{H^s}
\sum_{n=1}^\infty
\frac{(C_s\norm f_{H^s})^{n-1}}{n!}\le
\norm f_{H^s}e^{C_s\norm f_{H^s}},
\end{align*}
which proves \eqref{eq:exp1}. Similarly,
$$
e^f-f-1
=
\sum_{n=2}^\infty\frac{f^n}{n!},
$$
and hence
\begin{align*}
\norm{e^f-f-1}_{H^s}
&\le
\sum_{n=2}^\infty
\frac{C_s^{n-1}}{n!}\norm f_{H^s}^n\le
C_s\norm f_{H^s}^2
\sum_{n=2}^\infty
\frac{(C_s\norm f_{H^s})^{n-2}}{n!}\le
C_s\norm f_{H^s}^2
e^{C_s\norm f_{H^s}},
\end{align*}
which gives \eqref{eq:exp2}. For the difference estimate, write
$$
e^f-e^g
=
(f-g)\int_0^1e^{g+\theta(f-g)}\,\dd\theta.
$$
Since the constant function $1$ does not belong to $H^s$, it is convenient to write this as
$$
e^f-e^g
=
(f-g)
+
(f-g)\int_0^1
\bigl(e^{g+\theta(f-g)}-1\bigr)\,\dd\theta.
$$
Set
$$
w_\theta=g+\theta(f-g)
=(1-\theta)g+\theta f.
$$
Then
$$
\norm{w_\theta}_{H^s}
\le
\norm f_{H^s}+\norm g_{H^s},
\quad \text{for} \quad0\le\theta\le1.
$$
By \eqref{eq:exp1},
$$
\norm{e^{w_\theta}-1}_{H^s}
\le
\norm{w_\theta}_{H^s}
e^{C_s\norm{w_\theta}_{H^s}}
\le
(\norm f_{H^s}+\norm g_{H^s})
e^{C_s(\norm f_{H^s}+\norm g_{H^s})}.
$$
Using again the algebra property,
\begin{align*}
\norm{e^f-e^g}_{H^s}
&\le
\norm{f-g}_{H^s}
+
C_s\norm{f-g}_{H^s}
\int_0^1
\norm{e^{w_\theta}-1}_{H^s}\,\dd\theta\le
C_s
e^{C_s(\norm f_{H^s}+\norm g_{H^s})}
\norm{f-g}_{H^s},
\end{align*}
after increasing $C_s$ if necessary. This proves \eqref{eq:exp3}.

It remains to prove the tame estimate. Let $\sigma>s>1/2$. We use the standard estimate
\begin{equation}\label{eq:tame-product-exp}
\norm{fg}_{H^\sigma}
\le
C_{s,\sigma}
\left(
\norm f_{H^\sigma}\norm g_{H^s}
+
\norm f_{H^s}\norm g_{H^\sigma}
\right).
\end{equation}
An induction on $n$ gives
\begin{equation}\label{eq:power-tame}
\norm{f^n}_{H^\sigma}
\le
n C_{s,\sigma}^{\,n-1}
\norm f_{H^s}^{\,n-1}
\norm f_{H^\sigma},
\quad \text{for} \quad n\ge1.
\end{equation}
Indeed, the claim is immediate for $n=1$. Assuming it for $n$, we apply \eqref{eq:tame-product-exp} to $f^{n+1}=f^n f$ and use
$$
\norm{f^n}_{H^s}
\le
C_s^{n-1}\norm f_{H^s}^n,
$$
which gives \eqref{eq:power-tame}, after enlarging $C_{s,\sigma}$.

We may now sum the exponential series:
\begin{align*}
\norm{e^f-1}_{H^\sigma}
&\le
\sum_{n=1}^\infty
\frac{1}{n!}\norm{f^n}_{H^\sigma}\le
\norm f_{H^\sigma}
\sum_{n=1}^\infty
\frac{n C_{s,\sigma}^{\,n-1}\norm f_{H^s}^{\,n-1}}{n!}=
e^{C_{s,\sigma}\norm f_{H^s}}
\norm f_{H^\sigma}.
\end{align*}
Consequently,
$$
\norm{e^f-1}_{H^\sigma}
\le
C_{s,\sigma}(\norm f_{H^s})
\norm f_{H^\sigma},
$$
where $C_{s,\sigma}:[0,\infty)\to(0,\infty)$ is increasing. In particular, the estimate is uniform on bounded subsets of $H^s\cap H^\sigma$.
\end{proof}

Finally, along this work, we use a standard inhomogeneous Littlewood--Paley decomposition. Let $\chi\in C_c^\infty(\R)$ be even, with $\chi(\xi)=1$ for $|\xi|\le1$ and $\chi(\xi)=0$ for $|\xi|\ge2$. Set
$$
\varphi_0(\xi)=\chi(\xi),
$$
and, for $k\ge1$,
$$
\varphi_k(\xi)
=
\chi\left(2^{-k}\xi\right)
-
\chi\left(2^{-k+1}\xi\right).
$$
We define, for $k\ge0$
$$
\widehat{P_kf}(\xi)
=
\varphi_k(\xi)\widehat f(\xi).
$$
Thus $P_0$ localizes to low frequencies, while $P_k$ localizes to $|\xi|\sim2^k$. We also write
$$
\widehat{P_{\le k}f}(\xi)
=
\chi\left(2^{-k}\xi\right)\widehat f(\xi),
$$
and
$$
P_{>k}=I-P_{\le k}.
$$
With this convention, $f=\sum_{k\ge0}P_kf$ and
$$
\norm f_{H^s}^2
\sim
\sum_{k\ge0}2^{2sk}\norm{P_kf}_{L^2}^2.
$$
\begin{remark}
Throughout the paper we use without further comment the one-dimensional Sobolev embedding and the Kato--Ponce commutator estimate. We refer to \cite{LinaresPonce2015, MR2233925} for these classical tools. In particular, for $\sigma\ge0$,
$$
\norm{fg}_{H^\sigma}
\lesssim
\norm f_{L^\infty}\norm g_{H^\sigma}
+
\norm g_{L^\infty}\norm f_{H^\sigma},
$$
and, for $\sigma>0$,
$$
\norm{[J^\sigma,f]g}_{L^2}
\lesssim
\norm{f_x}_{L^\infty}\norm{J^{\sigma-1}g}_{L^2}
+
\norm f_{H^\sigma}\norm g_{L^\infty}.
$$

\end{remark}

\section{The Cole--Hopf transform}\label{sec:structure}

Applying $P_+$ to \eqref{eq:embo}, using Proposition \ref{prop:null}, and setting $Q:=P_+u$, we obtain
$$
Q_t-i Q_{xx}-i Q_x^2
+\gamma P_+\bigl[(e^u-1)u_x\bigr]=0,
$$
since, $P_+((P_-u_x)^2)=0$. Now, define
$$
U=e^Q-1,
$$
and
$$
a=a(u):=e^u-1.
$$
Since
$$
U_t=(1+U)Q_t,
\quad
U_x=(1+U)Q_x \quad \text{and} \quad U_{xx}=(1+U)(Q_{xx}+Q_x^2),
$$
while
$$
P_+(au_x)
=aP_+u_x+[P_+,a]u_x
=aQ_x+ [P_+,a]u_x,
$$
we obtain the following equation
\begin{equation}\label{eq:U}
U_t-i U_{xx}+\gamma aU_x
=-\gamma(1+U)R,
\end{equation}
where $R:=[P_+,a]u_x$. For real $u$
$$
P_+u=\frac12(u+i\Hh u),
$$
so
$$
|1+U|=|e^{P_+u}|=e^{u/2}.
$$
Consequently, whenever $1+U$ stays away from zero, we have
$$
u=2\ln|1+U|.
$$

\section{Commutators, composition and energy estimates}
\label{sec:commutator}

The key estimate is a one-sided form of the Bedrosian commutator phenomenon and may be viewed as a Sobolev analogue of the commutator bounds of Dawson, McGahagan and Ponce \cite{MR2383514}.

\begin{lemma}\label{lem:hardycomm}
Let $\sigma\ge0$ and $r>1/2$. For Schwartz functions $f,g$,
$$
\norm{[P_{\pm},f]g}_{H^\sigma}
\lesssim
\norm f_{H^\sigma}\norm g_{H^r}.
$$
Consequently, if
$s>3/2$, then
$$
\norm{[P_{\pm},f]\partial_xh}_{H^s}
\lesssim
\norm f_{H^s}\norm h_{H^s}.
$$
\end{lemma}

\begin{proof}
Since the identity operator commutes with multiplication and
$$
P_{\pm}=\frac12(I\pm i\Hh),
$$
the first estimate follows direct by  \cite[Theorem 1]{Oliver2018}. For the second estimate, take $r=s-1>\frac12$ and $g=\partial_xh$. Then
$$
\norm{\partial_xh}_{H^{s-1}}
\leq
\norm h_{H^s},
$$
and the conclusion follows from the first estimate.

\end{proof}
\begin{remark}
The point of Lemma \ref{lem:hardycomm} is that the commutator does not lose a derivative. The symbol of
$[P_+,f]g_x$ is supported where the input frequency $\eta$ and the output frequency $\xi$ have opposite signs. If $\theta=\xi-\eta$ is the frequency of $f$, then on this region
$$
|\theta|=|\xi|+|\eta|.
$$
Hence
$$
|\theta|\ge |\eta|
\quad\text{and}\quad
|\theta|\ge |\xi|.
$$
Therefore the derivative in $g_x$, which gives the factor $|\eta|$, can be controlled by the frequency of $f$. In this way, the derivative is effectively placed on the coefficient, and no derivative is lost in the estimate.
\end{remark}


We also record the invertibility of the Cole--Hopf   transform. The proof is a direct application of the Lemma \ref{lem:exp}, but we include the short argument because the nonvanishing condition on $|1+U|$ is important in our construction.

\begin{lemma}\label{lem:inverse}
Let $s>1/2$. The map
$$
u\mapsto e^{P_+u}-1
$$
is locally Lipschitz from $H^s(\R\to\R)$ to $H^s(\R\to\C)$. Conversely, if $U\in H^s(\R\to\C)$ and
$$
|1+U(x)|\ge c>0,
$$
then
$$
\mathcal L(U):=2\ln|1+U|\in H^s(\R\to\R),
$$
and, on sets where $|1+U|\ge c$ and $\norm U_{H^s}\le M$,
$$
\norm{\mathcal L(U)}_{H^s}\le C_{s,c,M}\norm U_{H^s}.
$$
The map $U\mapsto\mathcal L(U)$ is locally Lipschitz on the same sets. If $U=e^{P_+u}-1$ with $u$ real-valued, then $\mathcal L(U)=u$ and
$$
|1+U|=e^{u/2}.
$$
\end{lemma}

\begin{proof}
The forward statement follows from the boundedness of $P_+$ on $H^s$ and the Lemma \ref{lem:exp}. For the inverse, write $U=\rho+i\eta$. Then
$$
\mathcal L(U)
=\ln\bigl((1+\rho)^2+\eta^2\bigr).
$$
The assumption $|1+U|\ge c$ keeps $(\rho,\eta)$ a positive distance from the only singular point $(-1,0)$ of this smooth function. Since $H^s\hookrightarrow L^\infty$, the standard composition theorem and its local Lipschitz form give the stated estimates; see \cite{Taylor1991}. Finally, for real $u$,
$$
\Rea(P_+u)=\frac u2,
$$
which gives $|e^{P_+u}|=e^{u/2}$ and hence $\mathcal L(e^{P_+u}-1)=u$.
\end{proof}

We now prove the a priori estimate that drives the local theory. The key point is that every differentiated term in \eqref{eq:U} can be controlled at the same Sobolev level as $U$.

\begin{proposition}\label{prop:energy}
Let $s>3/2$, and let $u$ be a smooth real solution of \eqref{eq:embo} on an interval $I$. Assume that on $I$,
$$
\norm{U(t)}_{H^s}\le M, 
$$
and 
$$
|1+U(t,x)|\ge c>0.
$$
Then
\begin{equation}\label{eq: grownwall type estimate}
\frac{\dd}{\dd t}\norm U_{H^s}^2
\le
C_{s,\gamma,c,M}
\bigl(1+\norm U_{H^s}^2\bigr).
\end{equation}
\end{proposition}

\begin{proof}
Apply $J^s$ to \eqref{eq:U}, take the $L^2$ inner product with $J^sU$, and take real parts. Since $J^s$ commutes with $\partial_x$,
$$
\Rea\ip{i J^sU_{xx}}{J^sU}=0.
$$
For the transport term, write
$$
J^s(aU_x)=a\partial_xJ^sU+[J^s,a]U_x.
$$
For the first term, since $a$ is real-valued,
\begin{align*}
\Rea\ip{a\partial_xJ^sU}{J^sU}
&=\frac12\int_\R a\,\partial_x|J^sU|^2\dd x=-\frac12\int_\R a_x|J^sU|^2\dd x,
\end{align*}
so
$$
\left|
\Rea\ip{a\partial_xJ^sU}{J^sU}
\right|
\le
\frac12\norm{a_x}_{L^\infty}\norm U_{H^s}^2.
$$
For the second term, the Kato--Ponce inequality gives
$$
\norm{[J^s,a]U_x}_{L^2}
\lesssim
\norm{a_x}_{L^\infty}\norm U_{H^s}
+
\norm a_{H^s}\norm{U_x}_{L^\infty}.
$$

It remains to estimate the right-hand side. Lemma \ref{lem:hardycomm}, applied with $f=a$ and $g=u_x$, yields
\begin{equation}\label{eq:Rbound}
\norm{[P_+,a]u_x}_{H^s}
\lesssim
\norm a_{H^s}\norm u_{H^s}.
\end{equation}
Since $s>3/2$
$$
\norm{(1+U)[P_+,a]u_x}_{H^s}
\lesssim
(1+\norm U_{H^s})\norm{[P_+,a]u_x}_{H^s}.
$$

By Lemma \ref{lem:inverse}, the assumptions $|1+U|\ge c$ and $\norm U_{H^s}\le M$ imply
$$
\norm u_{H^s}\le C_{s,c,M}.
$$
Since $s>3/2$, we have $H^s\hookrightarrow W^{1,\infty},$ and the Lemma \ref{lem:exp} gives
$$
\norm a_{H^s}+\norm{a_x}_{L^\infty}
\le C_{s,c,M}.
$$
Hence \eqref{eq:Rbound} also gives
$$
\norm{[P_+,a]u_x}_{H^s}\le C_{s,c,M}.
$$
Combining these bounds with Cauchy--Schwarz inequality proves \eqref{eq: grownwall type estimate}.
\end{proof}

The same argument at the $H^\sigma$, with $\sigma \geq s$, level gives the higher-regularity estimate.

\begin{corollary}\label{prop:higher-energy}
Let $\sigma\ge s>3/2$ and let $u$ be a smooth solution of \eqref{eq:embo} on $I$. Assume
$$
\sup_{t\in I}\norm{u(t)}_{H^s}\le M.
$$
Then, on every subinterval on which $|1+U|$ is bounded from below by a positive constant,
$$
\frac{\dd}{\dd t}\norm U_{H^\sigma}^2
\le
C_{s,\sigma,\gamma,M}\norm U_{H^\sigma}^2.
$$
\end{corollary}\qed

\section{Existence and uniqueness of solutions}
\label{sec:existence}

We construct solutions for \eqref{eq:embo} by \emph{parabolic regularization} argument. This is the classical approach used, for example, in the Benjamin--Ono well-posedness theory for smooth solutions; see, for instance, \cite{Iorio1986,Iorio1991}.

\subsection*{Step 1: The regularized problem}

Let $\rho\in C_c^\infty(\R)$ be a standard real mollifier and set $\rho_\eps(x)=\eps^{-1}\rho(x/\eps)$ and
$$
u_0^\eps=\rho_\eps*u_0.
$$
Then $u_0^\eps\to u_0$ in $H^s$ and
$$
\sup_{0<\eps\le1}\norm{u_0^\eps}_{H^s}
\lesssim\norm{u_0}_{H^s}.
$$
For $\eps>0$, consider the following regularized Cauchy problem
\begin{equation}\label{eq:parabolic-reg}
\left\{
\begin{aligned}
&u_t^\eps+\Hh u_{xx}^\eps+u_x^\eps\Hh u_x^\eps
+\gamma(e^{u^\eps}-1)u_x^\eps
=\eps u_{xx}^\eps,\\
&u^\eps(0,x)=u_0^\eps.
\end{aligned}
\right.
\end{equation}
For each fixed $\eps>0$, a classical iteration argument applied to the equation \eqref{eq:parabolic-reg} yields a unique smooth solution on a nontrivial time interval. Set
$$
Q^\eps=P_+u^\eps,
\quad
U^\eps=e^{Q^\eps}-1,
\quad
a^\eps=e^{u^\eps}-1,
\quad \text{and} \quad
R^\eps=[P_+,a^\eps]u_x^\eps.
$$
Applying $P_+$ to \eqref{eq:parabolic-reg}, using Proposition \ref{prop:null}, and repeating the calculation leading to \eqref{eq:U}, we obtain
\begin{equation}\label{eq:Upar}
U_t^\eps-(i+\eps)U_{xx}^\eps
+\gamma a^\eps U_x^\eps
=
-\gamma(1+U^\eps)R^\eps
-\eps Q_x^\eps U_x^\eps.
\end{equation}

\subsection*{Step 2: Uniform estimates}

\begin{proposition}\label{prop:viscous}
Let $s>3/2$ and $u_0\in H^s$. There exist $T>0$, $C>0$, and $c_0>0$, depending only on $s$, $|\gamma|$, and an upper bound for $\norm{u_0}_{H^s}$, such that every sufficiently small $\eps>0$ solution of \eqref{eq:parabolic-reg} exists on $[0,T]$ and satisfies
\begin{equation}\label{eq:uniform-parabolic-bound}
\sup_{0\le t\le T}\norm{U^\eps(t)}_{H^s}^2
+
\eps\int_0^T\norm{U_x^\eps(t)}_{H^s}^2\dd t
\le C,
\end{equation}
and
$$
|1+U^\eps(t,x)|\ge c_0
$$
for all $(t,x)\in[0,T]\times\R$.
\end{proposition}

\begin{proof}
Since $u_0^\eps$ is real-valued,
$$
|1+U_0^\eps(x)|
=
|e^{P_+u_0^\eps(x)}|
=
e^{u_0^\eps(x)/2}.
$$
By Sobolev embedding and the uniform $H^s$ bound for $u_0^\eps$,
$$
\inf_{0<\eps\le1}\inf_x|1+U_0^\eps(x)|
\ge c_*>0.
$$
The composition estimate also gives
$$
\sup_{0<\eps\le1}\norm{U_0^\eps}_{H^s}<\infty.
$$
Choose $M>0$ such that
$$
\norm{U_0^\eps}_{H^s}\le \frac M2
$$
for every $0<\eps\le1$. By continuity, there is a maximal bootstrap interval on which
$$
\norm{U^\eps}_{H^s}\le M,
$$
and
$$
|1+U^\eps|\ge \frac{c_*}{2}.
$$
We show that both bounds improve on a time interval independent of $\eps$. Applying $J^s$ to \eqref{eq:Upar}, the Schr\"odinger term has zero real part, while the viscous term gives
$$
-\eps\Rea\ip{J^sU_{xx}^\eps}{J^sU^\eps}
=
\eps\norm{U_x^\eps}_{H^s}^2.
$$
The transport and commutator terms are estimated as in Proposition \ref{prop:energy}. On the bootstrap interval, they are bounded by
$$
C_{s,\gamma,c_*,M} \cdot
\bigl(1+\norm{U^\eps}_{H^s}^2\bigr).
$$

It remains to estimate the last term in \eqref{eq:Upar}. Since
$$
Q_x^\eps
=
\frac{U_x^\eps}{1+U^\eps},
$$
we have
$$
Q_x^\eps U_x^\eps
=
\frac{(U_x^\eps)^2}{1+U^\eps}.
$$
By the product and composition estimates, together with the bootstrap bounds,
$$
\norm{Q_x^\eps U_x^\eps}_{H^s}
\lesssim
\norm{U_x^\eps}_{L^\infty}
\norm{U_x^\eps}_{H^s}.
$$
Since $s>3/2$,
$$
\norm{U_x^\eps}_{L^\infty}
\lesssim
\norm{U^\eps}_{H^s}
\le M,
$$
and therefore
$$
\norm{Q_x^\eps U_x^\eps}_{H^s}
\lesssim
\norm{U_x^\eps}_{H^s}.
$$
Hence
$$
\begin{aligned}
\eps\left|
\ip{J^s(Q_x^\eps U_x^\eps)}{J^sU^\eps}
\right|
&\lesssim
\eps\norm{U_x^\eps}_{H^s}\norm{U^\eps}_{H^s} \le
\frac{\eps}{2}\norm{U_x^\eps}_{H^s}^2
+
C_{s,c_*,M}\eps\norm{U^\eps}_{H^s}^2.
\end{aligned}
$$

Combining the estimates,
$$
\frac{\dd}{\dd t}\norm{U^\eps}_{H^s}^2
+
\eps\norm{U_x^\eps}_{H^s}^2
\le
C_{s,\gamma,c_*,M}
\bigl(1+\norm{U^\eps}_{H^s}^2\bigr).
$$
By Gronwall's inequality, there exists $T_1>0$, independent of $\eps$, such that
$$
\norm{U^\eps(t)}_{H^s}
\le
\frac{3M}{4},
$$
for $0\le t\le T_1$, and
$$
\eps\int_0^{T_1}
\norm{U_x^\eps(t)}_{H^s}^2\dd t
\le C.
$$

It remains to improve the lower bound for $|1+U^\eps|$. From \eqref{eq:Upar} and the bootstrap bounds,
$$
\sup_{0<\eps\le1}
\norm{U_t^\eps}_{L^\infty(0,T_1;H^{s-2})}
\le C.
$$
Indeed, the linear terms are uniformly bounded in $H^{s-2}$, while the remaining terms are bounded in $H^{s-1}$, using $s-1>1/2$ and the lower bound for $|1+U^\eps|$. Hence
$$
\norm{U^\eps(t)-U_0^\eps}_{H^{s-2}}
\le Ct,
$$
while
$$
\norm{U^\eps(t)-U_0^\eps}_{H^s}
\le C.
$$
Choose
$$
\theta\in
\left(
\max\left\{\frac12,s-2\right\},
s
\right).
$$
Interpolation gives
$$
\norm{U^\eps(t)-U_0^\eps}_{H^\theta}
\lesssim
t^{(s-\theta)/2}.
$$
Since $\theta>1/2$, Sobolev embedding yields
$$
\norm{U^\eps(t)-U_0^\eps}_{L^\infty}
\lesssim
t^{(s-\theta)/2}.
$$
After decreasing $T\le T_1$ if necessary,
$$
\norm{U^\eps(t)-U_0^\eps}_{L^\infty}
\le
\frac{c_*}{4},
$$
and therefore
$$
|1+U^\eps(t,x)|
\ge
\frac{3c_*}{4}
>
\frac{c_*}{2}.
$$

Finally, fix an integer $m>s+2$. The same energy estimate at the $H^m$ level gives
$$
\frac{\dd}{\dd t}\norm{U^\eps}_{H^m}^2
+
\eps\norm{U_x^\eps}_{H^m}^2
\le
C_{s,m,\gamma,c_*,M}
\bigl(1+\norm{U^\eps}_{H^m}^2\bigr).
$$
Thus $\norm{U^\eps}_{H^m}$ stays finite on the bootstrap interval. The standard continuation criterion for the regularized parabolic problem then extends the smooth solution to $[0,T]$. Although the $H^m$ bound may depend on $\eps$, the time $T$ and the $H^s$ bounds are uniform. Therefore the bootstrap closes, and \eqref{eq:uniform-parabolic-bound} follows.
\end{proof}

\begin{remark}

If $u_0\in H^\sigma$ for some $\sigma>s$, the same parabolic estimate can be performed at the $H^\sigma$ level on the already fixed interval $[0,T]$. We spell out the only additional point. Set
$$
B^\eps=(1+U^\eps)^{-1}.
$$
The uniform $H^s$ bound and the lower bound for $|1+U^\eps|$ imply
$$
\norm{B^\eps-1}_{H^\sigma}
\le C_M\norm{U^\eps}_{H^\sigma}.
$$
Since
$$
Q_x^\eps U_x^\eps
=B^\eps(U_x^\eps)^2,
$$
the product estimate gives
$$
\norm{Q_x^\eps U_x^\eps}_{H^\sigma}
\le
C_M\norm{U_x^\eps}_{H^\sigma}
+C_M\norm{U^\eps}_{H^\sigma}.
$$
Here the $L^\infty$ norms of $U_x^\eps$ are controlled by the base $H^s$ bound. Consequently,
\begin{align*}
\eps\left|
\ip{J^\sigma(Q_x^\eps U_x^\eps)}{J^\sigma U^\eps}
\right|
&\le
\frac\eps2\norm{U_x^\eps}_{H^\sigma}^2
+C_M\norm{U^\eps}_{H^\sigma}^2.
\end{align*}
All remaining terms are controlled by the higher-regularity estimate of Proposition \ref{prop:higher-energy}. Thus Gronwall inequality yields
$$
\sup_{0\le t\le T}\norm{U^\eps(t)}_{H^\sigma}
\le C_{M,\sigma}\norm{U_0^\eps}_{H^\sigma}.
$$
Hence smooth initial data produce smooth inviscid solutions on the same lifespan determined by the base $H^s$ norm.
\end{remark}

\subsection*{Step 3: Compactness}

By Proposition \ref{prop:viscous} and Lemma \ref{lem:inverse},
$$
\sup_{0<\eps\le1}\norm{u^\eps}_{L^\infty(0,T;H^s)}<\infty.
$$
The regularized equation also gives
$$
\sup_{0<\eps\le1}\norm{u_t^\eps}_{L^\infty(0,T;H^{s-2})}<\infty,
$$
since the nonlinear terms are uniformly bounded in $H^{s-1}$. Hence, by Banach--Alaoglu, after passing to a subsequence,
$$
u^\eps\overset{*}{\rightharpoonup}u
\quad\text{in }L^\infty(0,T;H^s),
$$
and
$$
u_t^\eps\overset{*}{\rightharpoonup}u_t
\quad\text{in }L^\infty(0,T;H^{s-2}).
$$
By the standard local compactness argument, using a cutoff and Simon's criterion \cite[Corollary~4]{Simon1987}, for every bounded interval $K\subset\R$ and every $\delta>0$,
\begin{equation}\label{eq:localstrong}
u^\eps\to u
\quad\text{in }C([0,T];H^{s-\delta}(K)).
\end{equation}
In particular, $u(0)=u_0$ locally, since $u_0^\eps\to u_0$ in $H^s$.

Choose $\delta>0$ such that $s-\delta>1$. Then
$$
u_x^\eps\to u_x
\quad\text{in }L^2_{\mathrm{loc}}([0,T]\times\R),
$$
while
$$
\Hh u_x^\eps\rightharpoonup\Hh u_x
\quad\text{in }L^2((0,T)\times\R).
$$
Therefore, by the usual strong--weak convergence argument,
$$
u_x^\eps\Hh u_x^\eps
\to
u_x\Hh u_x
$$
in distributions. Since \eqref{eq:localstrong} also gives local uniform convergence of $u^\eps$ to $u$, we similarly obtain
$$
(e^{u^\eps}-1)u_x^\eps
\to
(e^u-1)u_x
$$
in distributions.

It remains to remove the artificial viscosity. From the uniform dissipative estimate and
$$
u_x^\eps
=
2\Rea\left(\frac{U_x^\eps}{1+U^\eps}\right),
$$
we obtain
$$
\eps\int_0^T\norm{u^\eps(t)}_{H^{s+1}}^2\dd t\lesssim1.
$$
Consequently,
$$
\norm{\eps u_{xx}^\eps}_{L^2(0,T;H^{s-1})}
\lesssim \sqrt{\eps}\to0.
$$
Passing to the limit in \eqref{eq:parabolic-reg}, we obtain a distributional solution of \eqref{eq:embo} such that
$$
u\in L^\infty(0,T;H^s),
$$
and
$$
u_t\in L^\infty(0,T;H^{s-2}).
$$

Finally, since $s-1>1/2$,
$$
F(u):=u_x\Hh u_x+\gamma(e^u-1)u_x
\in L^1([0,T];L^2).
$$
The standard linear theory then gives
$$
u(t)
=
S(t)u_0
-
\int_0^tS(t-t')F(u(t'))\dd t' \in L^2.
$$
Hence $u\in C([0,T];L^2),$ with $u(0)=u_0$ and, by local Lipschitz continuity of the exponential map on bounded $H^s$ sets,
$$
U=e^{P_+u}-1\in C([0,T];L^2).
$$
For later use we record that the transformed equation is valid at the regularity of the limiting solution, rather than only for smooth approximations.

\begin{lemma}\label{lem:rough-transform}
Let $s>3/2$ and let
$$
u\in L^\infty(I;H^s)\cap W^{1,\infty}(I;H^{s-2})
$$
be a solution of \eqref{eq:embo}. Then
$$
U\in L^\infty(I;H^s)\cap W^{1,\infty}(I;H^{s-2})
$$
and
\begin{equation}\label{eq:U-rough}
U_t-i U_{xx}+\gamma aU_x
=-\gamma(1+U)R
\end{equation}
in $\mathcal D'(I\times\R)$. In addition, if  $u\in C(I;L^2)$, then $U\in C(I;L^2)$.
\end{lemma}

\begin{proof}
Since $s>3/2$, the product map
$$
H^s\cdot H^{s-2}\to H^{s-2}
$$
is continuous, and the standard Sobolev composition theorem gives $U=e^Q-1\in L^\infty(I;H^s)$. By the usual chain rule, justified by a standard time-regularization argument,
$$
U_t=(1+U)Q_t
\in L^\infty(I;H^{s-2}).
$$
Similarly,
$$
U_x=(1+U)Q_x,
$$
and
$$
U_{xx}=(1+U)(Q_{xx}+Q_x^2),
$$
in distributions. Applying $P_+$ to \eqref{eq:embo} and using Proposition \ref{prop:null},
\begin{equation}\label{eq: P_+}
Q_t-iQ_{xx}-iQ_x^2+\gamma P_+(au_x)=0.
\end{equation}
Moreover,
$$
P_+(au_x)
=
aQ_x+[P_+,a]u_x
=
aQ_x+R.
$$
Multiplying the equation \eqref{eq: P_+} by $1+U$ and using the identities above gives \eqref{eq:U-rough}.

Finally, the last assertion follows from the local Lipschitz continuity of the map
$$
u\mapsto e^{P_+u}-1
$$
on bounded subsets of $H^s$, with respect to the $L^2$-metric.
\end{proof}

\begin{remark}
As is usual in parabolic regularization arguments, we construct the solution only for positive times. The negative-time solution is then obtained from the symmetry
$$
u(t,x)=\widetilde u(-t,-x),
$$
where $\widetilde u$ is the forward solution with initial data $u_0(-x)$.
\end{remark}

Finally, we prove establish the $L^2$-stability estimate at the level at which the rough solutions constructed above are already known to be strongly continuous

\begin{lemma}\label{lem:rough-stability}
Let $s>3/2$, and let
$$
u_j\in L^\infty(I;H^s)\cap C(I;L^2),
$$
with $j=1,2$, be real distributional solutions of \eqref{eq:embo} on an interval $I\ni 0$. Assume
$$
|1+U_j|\ge c>0,
$$
and
$$
\sup_{t\in I}\norm{U_j(t)}_{H^s}\le M,
$$
where $U_j=e^{P_+u_j}-1.$ Then, for $t\in I$,
\begin{equation}\label{eq:L2stab}
\norm{U_1(t)-U_2(t)}_{L^2}^2
\le
e^{C|t|}\norm{U_1(0)-U_2(0)}_{L^2}^2,
\end{equation}
where $C=C(s,|\gamma|,c,M)$.
\end{lemma}

\begin{proof}
For $j=1,2$, set
$$
a_j=e^{u_j}-1,
\quad
R_j=[P_+,a_j]u_{j,x},
$$
and
$$
W=U_1-U_2.
$$
By Lemma \ref{lem:rough-transform}, both transformed variables satisfy \eqref{eq:U-rough}. Subtracting the two equations gives, in distributions,
\begin{equation}\label{eq:W}
W_t-i W_{xx}+\gamma a_1W_x
=-\gamma(a_1-a_2)U_{2,x}
-\gamma WR_1
-\gamma(1+U_2)(R_1-R_2),
\end{equation}
where
\begin{equation}\label{eq:Rdiff}
R_1-R_2
=[P_+,a_1](u_1-u_2)_x
+[P_+,a_1-a_2]u_{2,x}.
\end{equation}

The inverse transform is pointwise Lipschitz on the set $|1+U_j|\ge c$, and the exponential is pointwise Lipschitz on bounded sets. Since the $H^s$ bounds imply uniform $L^\infty$ bounds, Lemma \ref{lem:inverse} and the mean-value theorem give
$$
\norm{u_1-u_2}_{L^2}
+
\norm{a_1-a_2}_{L^2}
\le
C_{s,c,M}\norm W_{L^2}.
$$
For the first term in \eqref{eq:Rdiff}, the first Calder\'on commutator estimate \cite{Calderon1965} yields
$$
\norm{[P_+,a_1](u_1-u_2)_x}_{L^2}
\lesssim
\norm{a_{1,x}}_{L^\infty}
\norm{u_1-u_2}_{L^2}
\le C\norm W_{L^2}.
$$
For the second term, Lemma \ref{lem:hardycomm}, with output regularity $0$ and input exponent $s-1>1/2$, gives
$$
\norm{[P_+,a_1-a_2]u_{2,x}}_{L^2}
\lesssim_s
\norm{a_1-a_2}_{L^2}
\norm{u_2}_{H^s}
\le C\norm W_{L^2}.
$$
Hence
$$
\norm{R_1-R_2}_{L^2}
\le C\norm W_{L^2}.
$$
In addition, \eqref{eq:Rbound} and Sobolev embedding imply
$$
\norm{R_1}_{L^\infty}
+
\norm{U_{2,x}}_{L^\infty}
+
\norm{a_{1,x}}_{L^\infty}
\le C.
$$
Therefore, if
$$
F
:=-\gamma(a_1-a_2)U_{2,x}
-\gamma WR_1
-\gamma(1+U_2)(R_1-R_2),
$$
then
\begin{equation}\label{eq:FdiffL2}
\norm F_{L^2}\le C\norm W_{L^2}.
\end{equation}

We now justify the energy estimate. Put $W_k=P_{\le k} W$. Applying $P_{\le k}$ to \eqref{eq:W} gives
$$
\partial_tW_k-i\partial_x^2W_k
+\gamma a_1\partial_xW_k
=
P_{k\leq k}F-\gamma[P_{\le k},a_1]\partial_x W.
$$
All terms in this equation belong to $L^1_{\rm loc}(I;L^2)$ after the frequency truncation, so taking the real $L^2$ inner product with $W_k$ is legitimate. The Schr\"odinger term is skew-adjoint, while
$$
\Rea\ip{a_1\partial_xW_k}{W_k}
=-\frac12\int_\R a_{1,x}|W_k|^2\dd x.
$$
The standard kernel estimate for a smooth Littlewood--Paley projector gives, uniformly in $k$,
$$
\norm{[P_{\le k},a_1]\partial_x W}_{L^2}
\lesssim
\norm{a_{1,x}}_{L^\infty}\norm W_{L^2}.
$$
Indeed, if $K_k$ is the convolution kernel of $P_{\le k}$, then
$$
[P_{\le k},a_1]\partial_x W(x)
=
\int K_k(x-y)\bigl(a_1(y)-a_1(x)\bigr)\partial_x W(y)\dd y,
$$
and integration by parts in $y$, together with the Lipschitz bound for $a_1$, gives the stated uniform estimate. Using also \eqref{eq:FdiffL2}, we obtain
$$
\frac{\dd}{\dd t}\norm{W_k(t)}_{L^2}^2
\le
C\norm{W(t)}_{L^2}^2.
$$
After integration from $0$ to $t$,
$$
\norm{W_k(t)}_{L^2}^2
\le
\norm{W_k(0)}_{L^2}^2
+C\int_0^{|t|}\norm{W(\tau)}_{L^2}^2\dd\tau,
$$
with the standard change of orientation for negative time. Since $W\in C(I;L^2)$, letting $k\to\infty$ gives
$$
\norm{W(t)}_{L^2}^2
\le
\norm{W(0)}_{L^2}^2
+C\int_0^{|t|}\norm{W(\tau)}_{L^2}^2\dd\tau.
$$
Gronwall's inequality proves \eqref{eq:L2stab}.
\end{proof}

In particular, two solutions in the above class with the same initial data coincide. This proves uniqueness.

\section{Continuous dependence and persistence of solutions}
\label{sec:stability}

\subsection{Frequency envelopes}

We now use frequency envelopes to improve the low-norm stability estimate in Lemma \ref{lem:rough-stability}. Frequency envelopes were introduced by Tao in the study of wave maps \cite{TaoWaveMaps2001}; see also Ifrim and Tataru \cite{IfrimTataru2023} for their role in quasilinear well-posedness, and \cite{MR4902804,IfrimTataru2024,IfrimPineauTataruTaylor2025} for further applications in different settings.

For the rest of this section, fix
$$
0<\delta<\frac12.
$$
For $\sigma\ge s$ and $f\in H^\sigma$, set, for $k\ge0$
$$
a_k^{(\sigma)}(f)
:=
2^{\sigma k}\norm{P_kf}_{L^2}.
$$

\begin{definition}[Admissible frequency envelope]
A positive sequence $c=(c_k)_{k\ge0}\in\ell^2$ is called a \textit{$\delta$-admissible $H^\sigma$ frequency envelope} for $f$ if
$$
a_k^{(\sigma)}(f)\le c_k
$$
and
$$
c_j\le 2^{\delta|j-k|}c_k
$$
for every $j,k\ge0$.
\end{definition}

A convenient choice is
\begin{equation}\label{eq:canonical-envelope}
c_k(f;\sigma)^2
=
2^{-2\delta k}\norm f_{H^\sigma}^2
+
\sum_{m\ge0}
2^{-2\delta|k-m|}
\bigl(a_m^{(\sigma)}(f)\bigr)^2.
\end{equation}
The first term prevents degeneracy of the envelope components and is useful when taking ratios.

\begin{lemma}\label{lem:canonical-envelope}
Let $\sigma\in\R$ and $0<\delta<1$. The sequence defined by \eqref{eq:canonical-envelope} satisfies
\begin{equation}\label{eq:canonical-basic}
a_k^{(\sigma)}(f)\le c_k(f;\sigma),
\qquad
c_j(f;\sigma)\le
2^{\delta|j-k|}c_k(f;\sigma),
\end{equation}
and
\begin{equation}\label{eq:canonical-l2}
\norm{c(f;\sigma)}_{\ell^2}
\sim
\norm f_{H^\sigma}.
\end{equation}
Moreover,
\begin{equation}\label{eq:canonical-lipschitz}
\norm{c(f;\sigma)-c(g;\sigma)}_{\ell^2}
\lesssim
\norm{f-g}_{H^\sigma}.
\end{equation}
\end{lemma}

\begin{proof}
The first estimate in \eqref{eq:canonical-basic} follows from the term $m=k$ in \eqref{eq:canonical-envelope}. The slow variation follows from
$$
2^{-\delta j}
\le
2^{\delta|j-k|}2^{-\delta k},
$$
and
$$
2^{-\delta|j-m|}
\le
2^{\delta|j-k|}
2^{-\delta|k-m|}.
$$
Thus
$$
c_j(f;\sigma)
\le
2^{\delta|j-k|}
c_k(f;\sigma).
$$

Summing \eqref{eq:canonical-envelope} in $k$ and changing the order of summation gives
$$
\sum_{k\ge0}c_k(f;\sigma)^2
\lesssim
\norm f_{H^\sigma}^2.
$$
The converse follows from
$$
a_k^{(\sigma)}(f)\le c_k(f;\sigma)
$$
and the Littlewood--Paley characterization of $H^\sigma$.

For \eqref{eq:canonical-lipschitz}, view $c_k(f;\sigma)$ as the norm in $\R\oplus\ell^2$ of
$$
\left(
2^{-\delta k}\norm f_{H^\sigma},
\,
\bigl(
2^{-\delta|k-m|}
a_m^{(\sigma)}(f)
\bigr)_{m\ge0}
\right).
$$
The reverse triangle inequality gives
$$
\begin{aligned}
|c_k(f;\sigma)-c_k(g;\sigma)|^2
&\le
2^{-2\delta k}\norm{f-g}_{H^\sigma}^2+
\sum_{m\ge0}
2^{-2\delta|k-m|}
2^{2\sigma m}
\norm{P_m(f-g)}_{L^2}^2.
\end{aligned}
$$
Summing in $k$ proves \eqref{eq:canonical-lipschitz}.
\end{proof}

Consequently, if $f_n\to f$ in $H^\sigma$, then
\begin{equation}\label{eq:uniform-envelope-tail}
\lim_{K\to\infty}
\sup_n
\sum_{k>K}c_k(f_n;\sigma)^2
=
0.
\end{equation}

For a fixed admissible envelope $c$, define
$$
\norm f_{H_c^\sigma}
:=
\sup_{k\ge0}
\frac{2^{\sigma k}\norm{P_kf}_{L^2}}{c_k}.
$$
The slow variation of $c$ allows the usual paraproduct estimates to be applied uniformly in $k$.

\begin{lemma}\label{lem:envelope-calculus}
Let $s>3/2$, $\sigma\ge s$, and let $c$ be a $\delta$-admissible envelope. Then
\begin{equation}\label{eq:envelope-product}
\norm{fg}_{H_c^\sigma}
\lesssim
\norm f_{H^s}\norm g_{H_c^\sigma}
+
\norm g_{H^s}\norm f_{H_c^\sigma}.
\end{equation}
If $F$ is smooth on a neighborhood of the range of $f$, with $F(0)=0$, then
\begin{equation}\label{eq:envelope-composition}
\norm{F(f)}_{H_c^\sigma}
\le
C_{F,\norm f_{H^s}}
\norm f_{H_c^\sigma}.
\end{equation}
Finally, for smooth $f,g$,
\begin{equation}\label{eq:envelope-dyadic-commutator}
\sup_{k\ge0}
\frac{2^{\sigma k}}{c_k}
\norm{[P_k,f]g_x}_{L^2}
\lesssim
\norm{f_x}_{L^\infty}\norm g_{H_c^\sigma}
+
\norm{g_x}_{L^\infty}\norm f_{H_c^\sigma}.
\end{equation}
\end{lemma}

\begin{proof}
These are standard consequences of the nonhomogeneous Bony decomposition, paralinearization, and dyadic commutator estimates; see \cite[Sections~2.8 and 2.10]{BahouriCheminDanchin2011}. We only indicate how the frequency envelope enters the estimates.

For the product estimate, write
$$
fg=T_fg+T_gf+R(f,g).
$$
In the low--high interactions, the output frequency $k$ forces the high input to have frequency $k+O(1)$, and hence the slow variation of $c$ gives
$$
\frac{2^{\sigma k}}{c_k}
\norm{P_kT_fg}_{L^2}
\lesssim
\norm f_{H^s}\norm g_{H_c^\sigma},
$$
and similarly with $f$ and $g$ interchanged. For the high--high interactions, the input frequencies satisfy
$j\ge k-C$. Using Bernstein and the envelope bound on one factor,
$$
\frac{2^{\sigma k}}{c_k}
\norm{P_kR(f,g)}_{L^2}
\lesssim
\norm f_{H^s}\norm g_{H_c^\sigma}
\sum_{j\ge k-C}
2^{\sigma k+k/2-(s+\sigma)j}
\frac{c_j}{c_k}.
$$
Since
$$
\frac{c_j}{c_k}
\lesssim
2^{\delta(j-k)},
$$
the sum is uniformly bounded. This proves \eqref{eq:envelope-product}.

The composition estimate follows in the same way from the standard paralinearization
$$
F(f)=T_{F'(f)}f+\mathcal R_F(f),
$$
where the remainder contains at least two comparable high frequencies. Finally, \eqref{eq:envelope-dyadic-commutator} follows from the standard dyadic commutator estimate, again using the slow variation of $c$ to compare neighboring envelope components.
\end{proof}

The special one-sided commutator appearing in our equation satisfies a stronger estimate. This is the dyadic version of Lemma \ref{lem:hardycomm}.

\begin{lemma}\label{lem:envelope-hardy}
Let $s>3/2$, $\sigma\ge s$, and let $c$ be $\delta$-admissible. Then
\begin{equation}\label{eq:envelope-hardy}
\norm{[P_{\pm},f]g_x}_{H_c^\sigma}
\lesssim_{s,\sigma,\delta}
\norm f_{H_c^\sigma}\norm g_{H^s}.
\end{equation}
\end{lemma}

\begin{proof}
We prove the estimate for $P_+$, the case $P_-$ is identical. Write
$$
f=\sum_{j \geq 0} P_jf,
\qquad
g=\sum_{\ell \geq 0} P_\ell g.
$$
The Fourier symbol of $[P_+,f]g_x$ vanishes unless the output frequency $\xi$ and the $g$-frequency $\eta$ have opposite signs. On the nonzero support,
$$
|\xi-\eta|
=
|\xi|+|\eta|.
$$
Thus, if the output has size $2^k$ and $P_\ell g$ has size $2^\ell$, then the $f$-frequency satisfies
\begin{equation}\label{eq:envelope-sign-geometry}
j=\max\{k,\ell\}+O(1).
\end{equation}

First suppose $\ell\le k+C$. Then $j=k+O(1)$, up to the comparable-frequency case, and Bernstein gives
$$
\norm{P_k[P_+,P_jf]\partial_xP_\ell g}_{L^2}
\lesssim
\norm{P_jf}_{L^2}
2^{3\ell/2}\norm{P_\ell g}_{L^2}.
$$
Hence
$$
\begin{aligned}
2^{\sigma k}
\sum_{\ell\le k+C}
\norm{P_k[P_+,P_jf]\partial_xP_\ell g}_{L^2}
&\lesssim
\norm f_{H_c^\sigma}c_k
\norm g_{H^s}
\sum_{\ell\le k+C}
2^{(3/2-s)\ell}\lesssim
\norm f_{H_c^\sigma}
\norm g_{H^s}c_k,
\end{aligned}
$$
since $s>3/2$.

Now suppose $\ell\ge k+C$. By \eqref{eq:envelope-sign-geometry},
$j=\ell+O(1)$. Since the output frequency is lower than both inputs,
$$
\norm{P_k(P_jf\partial_xP_\ell g)}_{L^2}
\lesssim
2^{k/2}2^\ell
\norm{P_jf}_{L^2}
\norm{P_\ell g}_{L^2}.
$$
Therefore
$$
\begin{aligned}
2^{\sigma k}
\sum_{\ell\ge k+C}
\norm{P_k(P_jf\partial_xP_\ell g)}_{L^2}
&\lesssim
\norm f_{H_c^\sigma}
\norm g_{H^s}
2^{\sigma k+k/2}
\sum_{\ell\ge k+C}
c_\ell
2^{-(\sigma+s-1)\ell}\\
&\lesssim
\norm f_{H_c^\sigma}
\norm g_{H^s}
c_k
2^{(3/2-s)k}
\sum_{\ell\ge k+C}
2^{-(\sigma+s-1-\delta)(\ell-k)}\\
&\lesssim
\norm f_{H_c^\sigma}
\norm g_{H^s}c_k.
\end{aligned}
$$
The comparable-frequency cases are covered by the same estimates. This proves \eqref{eq:envelope-hardy}.
\end{proof}

\subsection{Propagation of frequency envelopes}

We now show that an envelope for the initial data controls the same dyadic Sobolev profile throughout the lifespan.

\begin{proposition}\label{prop:propagate-envelope}
Let $s>3/2$, $\sigma\ge s$, and let $u$ be a smooth real-valued solution of \eqref{eq:embo} on an interval $I\ni0$. Assume
$$
\sup_{t\in I}\norm{u(t)}_{H^s}\le M.
$$
If $c$ is a $\delta$-admissible $H^\sigma$ frequency envelope for $u(0)$, then
\begin{equation}\label{eq:propagated-envelope-U}
\sup_{t\in I}
2^{\sigma k}\norm{P_kU(t)}_{L^2}
\le
C_{M,I,s,\sigma,\gamma,\delta}c_k
\end{equation}
for every $k\ge0$. Consequently,
\begin{equation}\label{eq:propagated-envelope-u}
\sup_{t\in I}
2^{\sigma k}\norm{P_ku(t)}_{L^2}
\le
C_{M,I,s,\sigma,\gamma,\delta}c_k.
\end{equation}
\end{proposition}

\begin{proof}
Since
$$
\norm{u}_{L_t^\infty H_x^s}\le M,
$$
Sobolev embedding gives
$$
\norm u_{L^\infty_{t,x}}
+
\norm{u_x}_{L^\infty_{t,x}}
\le C_M,
$$
and
$$
0<c_M
\le
|1+U|
=
e^{u/2}
\le C_M.
$$
By Lemma \ref{lem:envelope-calculus},
$$
\norm u_{H_c^\sigma}
+
\norm a_{H_c^\sigma}
\lesssim
\norm U_{H_c^\sigma},
$$
while at $t=0$,
$$
\norm{U(0)}_{H_c^\sigma}
\lesssim
\norm{u(0)}_{H_c^\sigma}
\lesssim_M1.
$$

Set
$$
b_k(t)
=
2^{\sigma k}\norm{P_kU(t)}_{L^2}
$$
and, for $t\ge0$,
$$
X(t)
:=
\sup_{0\le\tau\le t}
\sup_{k\ge0}
\frac{b_k(\tau)}{c_k}.
$$
Applying $P_k$ to \eqref{eq:U},
$$
\partial_tU_k-i\partial_x^2U_k
+\gamma a\partial_xU_k
=
-\gamma[P_k,a]U_x
-\gamma P_k((1+U)R).
$$
Since $a$ is real,
$$
\Rea\ip{a\partial_xU_k}{U_k}
=
-\frac12
\int_\R a_x|U_k|^2\dd x.
$$
Therefore
$$
\frac{\dd}{\dd t}\norm{U_k}_{L^2}
\lesssim_M
\norm{U_k}_{L^2}
+
\norm{[P_k,a]U_x}_{L^2}
+
\norm{P_k((1+U)R)}_{L^2}.
$$

By Lemma \ref{lem:envelope-calculus},
$$
2^{\sigma k}
\norm{[P_k,a]U_x}_{L^2}
\lesssim
X(t)c_k.
$$
Moreover, Lemma \ref{lem:envelope-hardy} gives
$$
\norm R_{H_c^\sigma}
\lesssim
\norm a_{H_c^\sigma}\norm u_{H^s}
\lesssim
X(t),
$$
Moreover,
$$
P_k((1+U)R)=P_kR+P_k(UR).
$$
Hence, by Lemma \ref{lem:envelope-hardy} and the product estimate,
$$
2^{\sigma k}\norm{P_k((1+U)R)}_{L^2}
\lesssim X(t)c_k.
$$
Thus
$$
\frac{\dd}{\dd t}b_k(t)
\le
C_Mb_k(t)
+
C_MX(t)c_k.
$$
Using the initial bound,
$$
\frac{b_k(t)}{c_k}
\le
C_M
+
C_M\int_0^tX(\tau)\dd\tau.
$$
Taking the supremum gives
$$
X(t)
\le
C_M
+
C_M\int_0^tX(\tau)\dd\tau.
$$
Gronwall's inequality proves \eqref{eq:propagated-envelope-U} for positive times. The negative-time case is identical.

Finally, applying the composition estimate to
$$
u=2\ln|1+U|
$$
gives \eqref{eq:propagated-envelope-u}.
\end{proof}

\begin{remark}

The same bounds hold for the solutions constructed in Section \ref{sec:existence}. Indeed, let
$$
u_{0,N}=P_{\le N}u_0.
$$
Then $u_{0,N}\to u_0$ in $H^s$, the sequence is uniformly bounded in $H^s$, and the same envelope $c$ controls $u_{0,N}$ up to a fixed constant depending only on the Littlewood--Paley decomposition. Since $u_{0,N}$ is smooth, the smooth-data construction gives a smooth solution $u_N$ on the common interval $[-T,T]$, and the uniform $H^s$ estimate gives
$$
\sup_N\sup_{|t|\le T}
\norm{u_N(t)}_{H^s}
<\infty.
$$
For $U_N:=e^{P_+u_N}-1,$ the Proposition \ref{prop:propagate-envelope} gives
$$
\sup_N\sup_{|t|\le T}
2^{\sigma k}
\norm{P_kU_N(t)}_{L^2}
\le
Cc_k.
$$
Since $\|U_N(0)- U(0)\|_{L^2} \to 0$, Lemma \ref{lem:rough-stability} yields
$$
U_N\to U
\quad\text{in }C([-T,T];L^2).
$$
Thus, for each fixed $k$,
$$
P_kU_N\to P_kU
\quad\text{in }C([-T,T];L^2),
$$
and passing to the limit gives
$$
\sup_{|t|\le T}
2^{\sigma k}
\norm{P_kU(t)}_{L^2}
\le
Cc_k.
$$
Applying the composition estimate to $u=2\ln|1+U|$ gives the corresponding bound for $u$.
\end{remark}

\subsection{Proof of continuous dependence}

Let $(u_{0,n})_{n\ge1}\subset H^s$ such that
$$
\norm{u_{0,n}-u_0}_{H^s}\to0,
$$
and let $u_n$ and $u$ be the corresponding solutions of \eqref{eq:embo} on a common interval $[-T,T]$. Set $U=e^{P_+u}-1$ and 
$$
U_n=e^{P_+u_n}-1.
$$
By \eqref{eq:L2stab}, we have
\begin{equation}\label{eq:Un-L2-convergence}
\sup_{|t|\le T}
\norm{U_n(t)-U(t)}_{L^2}
\to0.
\end{equation}

Denote $c^{(n)}=c(u_{0,n};s)$ and $c=c(u_0;s).$ By Lemma \ref{lem:canonical-envelope},
$$
\|c^{(n)}-c\|_{\ell^2}\to0,
$$
and therefore, by \eqref{eq:uniform-envelope-tail},
$$
\lim_{K\to\infty}
\sup_n
\sum_{k>K}(c_k^{(n)})^2
=
0.
$$

The propagation estimate, together with the approximation argument above, gives
$$
\sup_n\sup_{|t|\le T}
\norm{P_{>K}U_n(t)}_{H^s}^2
\lesssim
\sup_n
\sum_{k>K}(c_k^{(n)})^2,
$$
and
$$
\sup_{|t|\le T}
\norm{P_{>K}U(t)}_{H^s}^2
\lesssim
\sum_{k>K}c_k^2.
$$
Hence both high-frequency tails tend to zero uniformly as $K\to\infty$.

For the low frequencies,
$$
\norm{P_{\le K}(U_n-U)}_{H^s}
\lesssim
2^{sK}\norm{U_n-U}_{L^2}.
$$
Therefore
$$
\begin{aligned}
\norm{U_n-U}_{H^s}
&\le
\norm{P_{\le K}(U_n-U)}_{H^s}+
\norm{P_{>K}U_n}_{H^s}
+
\norm{P_{>K}U}_{H^s}.
\end{aligned}
$$
First choose $K$ large so that the last two terms are uniformly small, and then use \eqref{eq:Un-L2-convergence}. Thus
$$
\sup_{|t|\le T}
\norm{U_n(t)-U(t)}_{H^s}
\to0.
$$
Since $|1+U_n|$ and $|1+U|$ are uniformly bounded away from zero, Lemma \ref{lem:inverse} yields
$$
\sup_{|t|\le T}
\norm{u_n(t)-u(t)}_{H^s}
\to0.
$$
This proves continuous dependence.

The same decomposition gives strong continuity in time. Indeed,
$$
\sup_{|t|\le T}
\norm{P_{>K}U(t)}_{H^s}
\to0
$$
as $K\to\infty$, while for fixed $K$,
$$
\norm{P_{\le K}(U(t)-U(t_0))}_{H^s}
\lesssim
2^{sK}
\norm{U(t)-U(t_0)}_{L^2}
\to0
$$
as $t\to t_0$. Hence
$$
U\in C([-T,T];H^s),
$$
and Lemma \ref{lem:inverse} gives
$$
u\in C([-T,T];H^s).
$$
Finally, using the equation and the continuity of the nonlinear terms from $H^s$ to $H^{s-2}$,
$$
u\in C^1([-T,T];H^{s-2}).
$$
\qed

\subsection*{Persistence of higher regularity}

Let $\sigma>s>3/2$ and suppose $u_0\in H^\sigma$. Let $c=c(u_0;\sigma)$ be the canonical $H^\sigma$ envelope. Applying Proposition \ref{prop:propagate-envelope} at level $\sigma$ to smooth approximations and using the approximation argument above, we obtain
\begin{equation}\label{eq:higher-envelope}
\sup_{|t|\le T}
2^{\sigma k}
\norm{P_ku(t)}_{L^2}
\lesssim
c_k.
\end{equation}
Squaring and summing in $k$ give
$$
\sup_{|t|\le T}
\norm{u(t)}_{H^\sigma}
\lesssim
\norm c_{\ell^2}
\lesssim
\norm{u_0}_{H^\sigma}.
$$
Thus the higher regularity persists on the same lifespan determined by the lower $H^s$ norm.

Strong continuity in $H^\sigma$ follows from the same low--high decomposition. Indeed,
$$
\sup_{|t|\le T}
\norm{P_{>K}u(t)}_{H^\sigma}^2
\lesssim
\sum_{k>K}c_k^2
\to0.
$$
For fixed $K$, the already established continuity $u\in C([-T,T];H^s)$ gives
$$
\norm{P_{\le K}(u(t)-u(t_0))}_{H^\sigma}
\lesssim
2^{(\sigma-s)K}
\norm{u(t)-u(t_0)}_{H^s}
\to0
$$
as $t\to t_0$. Therefore $u\in C([-T,T];H^\sigma).$
\qed

\section{Proof of Theorem \ref{corolario s>sc}} \label{rem:gamma0}
If $\gamma=0$, equation \eqref{eq:U} reduces to the free Schr\"odinger equation
$$
U_t-iU_{xx}=0.
$$
Given $s>1/2$ and $u_0\in H^s$, we have $U_0:=e^{P_+u_0}-1\in H^s$ and
$$
|1+U_0|\ge e^{-\|P_+u_0\|_{L^\infty}}=:\varepsilon>0.
$$

Since $U(t):=e^{it\partial_x^2}U_0 \in C(\R, H^s)$ satisfies
$$
\lim_{t\to0}\|U(t)- U_0\|_{H^s}=0
$$
and $H^s\hookrightarrow L^\infty$, there exists $T>0$ such that
$$
\|U(t)-U_0\|_{L^\infty}<\frac{\varepsilon}{2}.
$$
for $t\in[-T,T]$. Therefore
$$
|1+U(t)|\ge |1+U_0| - |U(t)-U_0| \ge \frac{\varepsilon}{2},
$$
and the inverse $u(t)=2\ln|1+U(t)|$ is well defined for $t\in[-T,T]$. By the standard composition estimates in $H^s$, it follows that
$$
u\in C([-T,T];H^s).
$$
 Uniqueness and continuous dependence follow directly from the local Lipschitz continuity of the exponential and logarithmic maps in $H^s$ and the unitarity of the free Schr\"odinger group $\{e^{it\partial_x^2}\}_{t\in \R}$. Thus \eqref{eq:embo} is locally well-posed for every $s>1/2=s_c$,  proving Theorem \ref{corolario s>sc}.

 If $\gamma=0$, equation \eqref{eq:U} reduces to the free Schr\"odinger equation
$$
U_t-iU_{xx}=0.
$$
Given $s>1/2$ and $u_0\in H^s$, we have $U_0:=e^{P_+u_0}-1\in H^s$ and
$$
|1+U_0|\ge e^{-\|P_+u_0\|_{L^\infty}}=:\varepsilon>0.
$$
Since $U(t):=e^{it\partial_x^2}U_0\in C(\R;H^s)$ and $H^s\hookrightarrow L^\infty$, there exists $T>0$ such that
$$
\|U(t)-U_0\|_{L^\infty}<\frac{\varepsilon}{2},\qquad |t|\le T.
$$
Hence $|1+U(t)|\ge\varepsilon/2$. Moreover, $U_0$ has positive Fourier support, which is preserved by the Schr\"odinger flow. After decreasing $T$ if necessary,
$$
\big\|e^{-P_+u_0}(U(t)-U_0)\big\|_{L^\infty}<\frac12.
$$
Thus we may define
$$
Q(t):=P_+u_0+\log\!\left(1+e^{-P_+u_0}(U(t)-U_0)\right).
$$
By the power-series expansion of the logarithm, $Q$ has positive Fourier support and $e^Q=1+U$. Hence $u:=2\Rea Q=2\ln|1+U|$ satisfies $P_+u=Q$, and reversing the computation leading to \eqref{eq:U} shows that $u$ solves \eqref{eq:embo}. By the standard composition estimates in $H^s$, it follows that
$$
u\in C([-T,T];H^s).
$$
 Uniqueness and continuous dependence follow directly from the local Lipschitz continuity of the exponential and logarithmic maps in $H^s$ and the unitarity of the free Schr\"odinger group $\{e^{it\partial_x^2}\}_{t\in \R}$. Thus \eqref{eq:embo} is locally well-posed for every $s>1/2=s_c$,  proving Theorem \ref{corolario s>sc}.

\appendix

\section{Failure of boundedness of the quadratic Duhamel operator}\label{sec:c2}

We now prove Theorem \ref{thm:c2}. The argument is a high--low second-iterate test, as in the ill-posedness mechanisms for the Benjamin--Ono equation; see \cite{MolinetSautTzvetkov2001,KochTzvetkov2005}.

Recall the operator $\mathcal B(t)$ defined in \eqref{operator B}. It is enough to construct a real sequence $\varphi_N$ such that
$$
\norm{\varphi_N}_{H^s}\sim1,
$$
while
$$
\norm{\mathcal B(t)(\varphi_N,\varphi_N)}_{H^s}\to\infty.
$$
Fix $0<\delta<1$ and set $\eta=N^{-1-\delta}.$ Define
$$
I_\ell=[\eta/2,\eta],
\quad
I_h=[-N,-N+\eta],
$$
and let
$$
\widehat{\varphi_N}(\xi)
=
\eta^{-1/2}\mathbf 1_{I_\ell}(\xi)
+
N^{-s}\eta^{-1/2}\mathbf 1_{I_h}(\xi).
$$
 Extend by
$$
\widehat{\varphi_N}(-\xi)=\overline{\widehat{\varphi_N}(\xi)}
$$
so that $\varphi_N$ is real-valued. Since the four blocks are disjoint $\norm{\varphi_N}_{H^s}\sim1$ uniformly in $N$.

We consider the high--low interaction
$$
I_h+I_\ell
\subset
[-N+\eta/2,-N+2\eta].
$$
Since the two input frequencies have opposite signs, Proposition \ref{prop:null} implies that the quadratic symbol of $u_x\Hh u_x$ vanishes on this interaction. Hence only the term $\gamma uu_x$ contributes.

For $\xi_1\in I_\ell$ and $\xi_2\in I_h$, one has  $\xi_2<0<\xi_1$, and $\xi_1+\xi_2<0$. Therefore
$$
\Omega(\xi_1,\xi_2)
:=
\xi_1|\xi_1|+\xi_2|\xi_2|
-(\xi_1+\xi_2)|\xi_1+\xi_2|
=
2\xi_1(\xi_1+\xi_2),
$$
and consequently
$$
|\Omega(\xi_1,\xi_2)|
\sim N\eta
=
N^{-\delta}.
$$
Thus, for every fixed $t\ne0$,
$$
\frac{e^{-it\Omega}-1}{\Omega}
\to -it
$$
uniformly on this interaction.

On a subinterval of $I_h+I_\ell$ of length comparable to $\eta$, the convolution of the two blocks satisfies
$$
\left|
\int
\widehat{\varphi_N}(\xi_1)
\widehat{\varphi_N}(\xi-\xi_1)
\,d\xi_1
\right|
\gtrsim
\eta\,
\eta^{-1/2}
N^{-s}\eta^{-1/2}
=
N^{-s}.
$$
Moreover, the transport multiplier satisfies
$$
|\gamma(\xi_1+\xi_2)|
\sim |\gamma|N.
$$
Hence, on an interval of length comparable to $\eta$,
$$
\left|
\mathcal{F}\left[\mathcal B(t)(\varphi_N,\varphi_N)\right](\xi)
\right|
\gtrsim_t
N^{1-s}.
$$
Since $|\xi|\sim N$ on this interval,
$$
\norm{\mathcal B(t)(\varphi_N,\varphi_N)}_{H^s}
\gtrsim_t
N^s N^{1-s}\eta^{1/2}
=
N\eta^{1/2}
=
N^{(1-\delta)/2}.
$$
Therefore
\begin{equation}\label{eq:second-blowup}
\norm{\mathcal B(t)(\varphi_N,\varphi_N)}_{H^s}
\gtrsim_t
N^{(1-\delta)/2}
\to\infty.
\end{equation}

Finally, if the Duhamel map were $C^2$ at the origin as a map into $C([-T,T];H^s)$, then differentiating the Duhamel formula at zero would give
$$
D^2\Phi(0)[\varphi,\psi](t)
=
-2\mathcal B(t)(\varphi,\psi).
$$
The second derivative would therefore define a bounded bilinear map on $H^s\times H^s$, contradicting \eqref{eq:second-blowup}. This completes the proof.

\begin{remark}
When $\gamma=0$, the opposite-sign interaction disappears because the quadratic symbol of $u_x\Hh u_x$ vanishes identically in that region. In this case the equation is considerably simpler and, as observed in the Section \ref{rem:gamma0}, can be explicitly linearized through the Cole--Hopf transform.
\end{remark}

\bibliographystyle{amsplain}
\bibliography{references}

@article {Calderon1965,
    AUTHOR = {Calder\'{o}n, A.-P.},
     TITLE = {Commutators of singular integral operators},
   JOURNAL = {Proc. Nat. Acad. Sci. U.S.A.},
  FJOURNAL = {Proceedings of the National Academy of Sciences of the United
              States of America},
    VOLUME = {53},
      YEAR = {1965},
     PAGES = {1092--1099},
      ISSN = {0027-8424},
   MRCLASS = {47.70},
  MRNUMBER = {177312},
MRREVIEWER = {R.\ T.\ Seeley},
       DOI = {10.1073/pnas.53.5.1092},
       URL = {https://doi.org/10.1073/pnas.53.5.1092},
}

@article{Guo2024,
  AUTHOR = {Guo, Ning},
  TITLE = {Integrability and inverse scattering transform of the modified {Benjamin--Ono} equation},
  JOURNAL = {Phys. Scr.},
  VOLUME = {99},
  YEAR = {2024},
  NUMBER = {11},
  PAGES = {Paper No. 115208},
  DOI = {10.1088/1402-4896/ad7e55}
}

@article {IonescuKenig2007,
    AUTHOR = {Ionescu, Alexandru D. and Kenig, Carlos E.},
     TITLE = {Global well-posedness of the {B}enjamin-{O}no equation in
              low-regularity spaces},
   JOURNAL = {J. Amer. Math. Soc.},
  FJOURNAL = {Journal of the American Mathematical Society},
    VOLUME = {20},
      YEAR = {2007},
    NUMBER = {3},
     PAGES = {753--798},
      ISSN = {0894-0347,1088-6834},
   MRCLASS = {35Q53 (35B30 76B55 76B70)},
  MRNUMBER = {2291918},
MRREVIEWER = {Grzegorz\ Karch},
       DOI = {10.1090/S0894-0347-06-00551-0},
       URL = {https://doi.org/10.1090/S0894-0347-06-00551-0},
}

@article {KenigTakaoka2006,
    AUTHOR = {Kenig, Carlos E. and Takaoka, Hideo},
     TITLE = {Global wellposedness of the modified {B}enjamin-{O}no equation
              with initial data in {$H^{1/2}$}},
   JOURNAL = {Int. Math. Res. Not.},
  FJOURNAL = {International Mathematics Research Notices},
      YEAR = {2006},
     PAGES = {Art. ID 95702, 44},
      ISSN = {1073-7928,1687-0247},
   MRCLASS = {35Q53 (35B30 76B03 76B15)},
  MRNUMBER = {2219229},
MRREVIEWER = {Mariana\ H\u{a}r\u{a}gu\c{s}},
       DOI = {10.1155/IMRN/2006/95702},
       URL = {https://doi.org/10.1155/IMRN/2006/95702},
}

@article {KillipLaurensVisan2024,
    AUTHOR = {Killip, Rowan and Laurens, Thierry and Vi\c{s}an, Monica},
     TITLE = {Sharp well-posedness for the {B}enjamin-{O}no equation},
   JOURNAL = {Invent. Math.},
  FJOURNAL = {Inventiones Mathematicae},
    VOLUME = {236},
      YEAR = {2024},
    NUMBER = {3},
     PAGES = {999--1054},
      ISSN = {0020-9910,1432-1297},
   MRCLASS = {35C08},
  MRNUMBER = {4743514},
MRREVIEWER = {Francisco\ J.\ Vielma Leal},
       DOI = {10.1007/s00222-024-01250-8},
       URL = {https://doi.org/10.1007/s00222-024-01250-8},
}

@article {KochTzvetkov2005,
    AUTHOR = {Koch, H. and Tzvetkov, N.},
     TITLE = {Nonlinear wave interactions for the {B}enjamin-{O}no equation},
   JOURNAL = {Int. Math. Res. Not.},
  FJOURNAL = {International Mathematics Research Notices},
      YEAR = {2005},
    NUMBER = {30},
     PAGES = {1833--1847},
      ISSN = {1073-7928,1687-0247},
   MRCLASS = {35Q53 (35K15 35K55 76B55)},
  MRNUMBER = {2172940},
MRREVIEWER = {E.\ I.\ Galakhov},
       DOI = {10.1155/IMRN.2005.1833},
       URL = {https://doi.org/10.1155/IMRN.2005.1833},
}

@article {Miura1968,
    AUTHOR = {Miura, Robert M.},
     TITLE = {Korteweg-de {V}ries equation and generalizations. {I}. {A}
              remarkable explicit nonlinear transformation},
   JOURNAL = {J. Mathematical Phys.},
  FJOURNAL = {Journal of Mathematical Physics},
    VOLUME = {9},
      YEAR = {1968},
     PAGES = {1202--1204},
      ISSN = {0022-2488,1089-7658},
   MRCLASS = {35.36 (82.00)},
  MRNUMBER = {252825},
       DOI = {10.1063/1.1664700},
       URL = {https://doi.org/10.1063/1.1664700},
}

@article {MolinetRibaud2004,
    AUTHOR = {Molinet, Luc and Ribaud, Francis},
     TITLE = {Well-posedness results for the generalized {B}enjamin-{O}no
              equation with small initial data},
   JOURNAL = {J. Math. Pures Appl. (9)},
  FJOURNAL = {Journal de Math\'{e}matiques Pures et Appliqu\'{e}es.
              Neuvi\`eme S\'{e}rie},
    VOLUME = {83},
      YEAR = {2004},
    NUMBER = {2},
     PAGES = {277--311},
      ISSN = {0021-7824},
   MRCLASS = {35Q55 (35B30 76B03 76B55)},
  MRNUMBER = {2038121},
MRREVIEWER = {J\"{u}rgen\ Scheurle},
       DOI = {10.1016/j.matpur.2003.11.005},
       URL = {https://doi.org/10.1016/j.matpur.2003.11.005},
}

@article {MolinetSautTzvetkov2001,
    AUTHOR = {Molinet, L. and Saut, J. C. and Tzvetkov, N.},
     TITLE = {Ill-posedness issues for the {B}enjamin-{O}no and related
              equations},
   JOURNAL = {SIAM J. Math. Anal.},
  FJOURNAL = {SIAM Journal on Mathematical Analysis},
    VOLUME = {33},
      YEAR = {2001},
    NUMBER = {4},
     PAGES = {982--988},
      ISSN = {0036-1410,1095-7154},
   MRCLASS = {35Q53 (35B30)},
  MRNUMBER = {1885293},
       DOI = {10.1137/S0036141001385307},
       URL = {https://doi.org/10.1137/S0036141001385307},
}

@article {Oliver2018,
    AUTHOR = {Oliver, Marcel},
     TITLE = {An estimate on the {B}edrosian commutator in {S}obolev space},
   JOURNAL = {J. Inequal. Appl.},
  FJOURNAL = {Journal of Inequalities and Applications},
      YEAR = {2018},
     PAGES = {Paper No. 348, 4},
      ISSN = {1029-242X},
   MRCLASS = {42A85 (26D10 46E35)},
  MRNUMBER = {3892484},
MRREVIEWER = {Liparit\ Tepoyan},
       DOI = {10.1186/s13660-018-1940-3},
       URL = {https://doi.org/10.1186/s13660-018-1940-3},
}

@article {ScoufisCosgrove1995,
    AUTHOR = {Scoufis, George and Cosgrove, Christopher M.},
     TITLE = {On the initial value problem for the modified {B}enjamin-{O}no
              equation},
   JOURNAL = {J. Math. Phys.},
  FJOURNAL = {Journal of Mathematical Physics},
    VOLUME = {36},
      YEAR = {1995},
    NUMBER = {10},
     PAGES = {5753--5759},
      ISSN = {0022-2488,1089-7658},
   MRCLASS = {35Q53},
  MRNUMBER = {1353114},
       DOI = {10.1063/1.531284},
       URL = {https://doi.org/10.1063/1.531284},
}

@article {Tao2004,
    AUTHOR = {Tao, Terence},
     TITLE = {Global well-posedness of the {B}enjamin-{O}no equation in
              {$H^1({\bf R})$}},
   JOURNAL = {J. Hyperbolic Differ. Equ.},
  FJOURNAL = {Journal of Hyperbolic Differential Equations},
    VOLUME = {1},
      YEAR = {2004},
    NUMBER = {1},
     PAGES = {27--49},
      ISSN = {0219-8916,1793-6993},
   MRCLASS = {35Q53 (35B30 76B03)},
  MRNUMBER = {2052470},
MRREVIEWER = {Grzegorz\ Karch},
       DOI = {10.1142/S0219891604000032},
       URL = {https://doi.org/10.1142/S0219891604000032},
}

@article {TaoWaveMaps2001,
    AUTHOR = {Tao, Terence},
     TITLE = {Global regularity of wave maps. {II}. {S}mall energy in two
              dimensions},
   JOURNAL = {Comm. Math. Phys.},
  FJOURNAL = {Communications in Mathematical Physics},
    VOLUME = {224},
      YEAR = {2001},
    NUMBER = {2},
     PAGES = {443--544},
      ISSN = {0010-3616,1432-0916},
   MRCLASS = {58J45 (35B60 35B65 35L15 58J47)},
  MRNUMBER = {1869874},
MRREVIEWER = {Joachim\ Krieger},
       DOI = {10.1007/PL00005588},
       URL = {https://doi.org/10.1007/PL00005588},
}

@article {IfrimTataru2023,
    AUTHOR = {Ifrim, Mihaela and Tataru, Daniel},
     TITLE = {Local well-posedness for quasi-linear problems: a primer},
   JOURNAL = {Bull. Amer. Math. Soc. (N.S.)},
  FJOURNAL = {American Mathematical Society. Bulletin. New Series},
    VOLUME = {60},
      YEAR = {2023},
    NUMBER = {2},
     PAGES = {167--194},
      ISSN = {0273-0979,1088-9485},
   MRCLASS = {35L45 (35L50 35L60)},
  MRNUMBER = {4557379},
       DOI = {10.1090/bull/1775},
       URL = {https://doi.org/10.1090/bull/1775},
}

@article {Iorio1986,
    AUTHOR = {I\'{o}rio, Jr., Rafael Jos\'{e}},
     TITLE = {On the {C}auchy problem for the {B}enjamin-{O}no equation},
   JOURNAL = {Comm. Partial Differential Equations},
  FJOURNAL = {Communications in Partial Differential Equations},
    VOLUME = {11},
      YEAR = {1986},
    NUMBER = {10},
     PAGES = {1031--1081},
      ISSN = {0360-5302,1532-4133},
   MRCLASS = {35B40 (35K55 35Q20)},
  MRNUMBER = {847994},
MRREVIEWER = {Bjorn\ Birnir},
       DOI = {10.1080/03605308608820456},
       URL = {https://doi.org/10.1080/03605308608820456},
}

@article {Iorio1991,
    AUTHOR = {I\'{o}rio, Jr., Rafael Jos\'{e}},
     TITLE = {The {B}enjamin-{O}no equation in weighted {S}obolev spaces},
   JOURNAL = {J. Math. Anal. Appl.},
  FJOURNAL = {Journal of Mathematical Analysis and Applications},
    VOLUME = {157},
      YEAR = {1991},
    NUMBER = {2},
     PAGES = {577--590},
      ISSN = {0022-247X,1096-0813},
   MRCLASS = {35Q53},
  MRNUMBER = {1112336},
       DOI = {10.1016/0022-247X(91)90108-C},
       URL = {https://doi.org/10.1016/0022-247X(91)90108-C},
}

@article {Simon1987,
    AUTHOR = {Simon, Jacques},
     TITLE = {Compact sets in the space {$L^p(0,T;B)$}},
   JOURNAL = {Ann. Mat. Pura Appl. (4)},
  FJOURNAL = {Annali di Matematica Pura ed Applicata. Serie Quarta},
    VOLUME = {146},
      YEAR = {1987},
     PAGES = {65--96},
      ISSN = {0003-4622},
   MRCLASS = {46E40 (46E30)},
  MRNUMBER = {916688},
MRREVIEWER = {James\ Bell\ Cooper},
       DOI = {10.1007/BF01762360},
       URL = {https://doi.org/10.1007/BF01762360},
}

@book {Taylor1991,
    AUTHOR = {Taylor, Michael E.},
     TITLE = {Pseudodifferential operators and nonlinear {PDE}},
    SERIES = {Progress in Mathematics},
    VOLUME = {100},
 PUBLISHER = {Birkh\"{a}user Boston, Inc., Boston, MA},
      YEAR = {1991},
     PAGES = {213},
      ISBN = {0-8176-3595-5},
   MRCLASS = {35Sxx (35G20 46E35 46F99 47G30 58G15)},
  MRNUMBER = {1121019},
MRREVIEWER = {William\ Margulies},
       DOI = {10.1007/978-1-4612-0431-2},
       URL = {https://doi.org/10.1007/978-1-4612-0431-2},
}

@book {BahouriCheminDanchin2011,
    AUTHOR = {Bahouri, Hajer and Chemin, Jean-Yves and Danchin, Rapha\"{e}l},
     TITLE = {Fourier analysis and nonlinear partial differential equations},
    SERIES = {Grundlehren der mathematischen Wissenschaften [Fundamental
              Principles of Mathematical Sciences]},
    VOLUME = {343},
 PUBLISHER = {Springer, Heidelberg},
      YEAR = {2011},
     PAGES = {xvi+523},
      ISBN = {978-3-642-16829-1},
   MRCLASS = {35-02 (35L72 35Q30 42-02 42B37 76B03 76D03 76N10)},
  MRNUMBER = {2768550},
MRREVIEWER = {Peter\ R.\ Massopust},
       DOI = {10.1007/978-3-642-16830-7},
       URL = {https://doi.org/10.1007/978-3-642-16830-7},
}

@article{MR3397003,
  AUTHOR = {Molinet, Luc and Vento, St\'{e}phane},
  TITLE = {Improvement of the energy method for strongly nonresonant dispersive equations and applications},
  JOURNAL = {Anal. PDE},
  VOLUME = {8},
  YEAR = {2015},
  NUMBER = {6},
  PAGES = {1455--1495},
  MRNUMBER = {3397003},
  DOI = {10.2140/apde.2015.8.1455}
}

@article{MR3906854,
  AUTHOR = {Molinet, Luc and Pilod, Didier and Vento, St\'{e}phane},
  TITLE = {On well-posedness for some dispersive perturbations of {B}urgers' equation},
  JOURNAL = {Ann. Inst. H. Poincar\'e C Anal. Non Lin\'eaire},
  VOLUME = {35},
  YEAR = {2018},
  NUMBER = {7},
  PAGES = {1719--1756},
  MRNUMBER = {3906854},
  DOI = {10.1016/j.anihpc.2017.12.004}
}

@misc{MolinetPilodVento2025,
  AUTHOR = {Molinet, Luc and Pilod, Didier and Vento, St{\'e}phane},
  TITLE = {Low regularity well-posedness of nonlocal dispersive perturbations of {B}urgers' equation},
  YEAR = {2025},
  NOTE = {Preprint, arXiv:2506.17801}
}

@article{MR5027534,
  AUTHOR = {Freire, Ricardo and Santos, Thyago S. R.},
  TITLE = {The {C}auchy problem for nonlinear dispersive models of long internal waves in the presence of the {C}oriolis force},
  JOURNAL = {Nonlinear Anal. Real World Appl.},
  VOLUME = {92},
  YEAR = {2026},
  PAGES = {Paper No. 104620, 25 pp.},
  MRNUMBER = {5027534},
  DOI = {10.1016/j.nonrwa.2026.104620}
}

@incollection {Kato1983,
    AUTHOR = {Kato, Tosio},
     TITLE = {On the {C}auchy problem for the (generalized) {K}orteweg-de
              {V}ries equation},
 BOOKTITLE = {Studies in applied mathematics},
    SERIES = {Adv. Math. Suppl. Stud.},
    VOLUME = {8},
     PAGES = {93--128},
 PUBLISHER = {Academic Press, New York},
      YEAR = {1983},
      ISBN = {0-12-305480-X},
   MRCLASS = {35Q20},
  MRNUMBER = {759907},
MRREVIEWER = {Amy\ Cohen},
}

@article {KenigPonceVega1993,
    AUTHOR = {Kenig, Carlos E. and Ponce, Gustavo and Vega, Luis},
     TITLE = {Well-posedness and scattering results for the generalized
              {K}orteweg-de {V}ries equation via the contraction principle},
   JOURNAL = {Comm. Pure Appl. Math.},
  FJOURNAL = {Communications on Pure and Applied Mathematics},
    VOLUME = {46},
      YEAR = {1993},
    NUMBER = {4},
     PAGES = {527--620},
      ISSN = {0010-3640,1097-0312},
   MRCLASS = {35Q53 (47H15 47N20)},
  MRNUMBER = {1211741},
MRREVIEWER = {Benno\ Fuchssteiner},
       DOI = {10.1002/cpa.3160460405},
       URL = {https://doi.org/10.1002/cpa.3160460405},
}

@article {Bourgain1993KdV,
    AUTHOR = {Bourgain, J.},
     TITLE = {Fourier transform restriction phenomena for certain lattice
              subsets and applications to nonlinear evolution equations.
              {II}. {T}he {K}d{V}-equation},
   JOURNAL = {Geom. Funct. Anal.},
  FJOURNAL = {Geometric and Functional Analysis},
    VOLUME = {3},
      YEAR = {1993},
    NUMBER = {3},
     PAGES = {209--262},
      ISSN = {1016-443X,1420-8970},
   MRCLASS = {35Q55},
  MRNUMBER = {1215780},
MRREVIEWER = {Yun\ Mei\ Chen},
       DOI = {10.1007/BF01895688},
       URL = {https://doi.org/10.1007/BF01895688},
}

@article {KenigPonceVega1996,
    AUTHOR = {Kenig, Carlos E. and Ponce, Gustavo and Vega, Luis},
     TITLE = {A bilinear estimate with applications to the {K}d{V} equation},
   JOURNAL = {J. Amer. Math. Soc.},
  FJOURNAL = {Journal of the American Mathematical Society},
    VOLUME = {9},
      YEAR = {1996},
    NUMBER = {2},
     PAGES = {573--603},
      ISSN = {0894-0347,1088-6834},
   MRCLASS = {35Q53 (35Bxx)},
  MRNUMBER = {1329387},
MRREVIEWER = {F.\ Pempinelli},
       DOI = {10.1090/S0894-0347-96-00200-7},
       URL = {https://doi.org/10.1090/S0894-0347-96-00200-7},
}

@article {KillipVisan2019,
    AUTHOR = {Killip, Rowan and Vi\c{s}an, Monica},
     TITLE = {Kd{V} is well-posed in {$H^{-1}$}},
   JOURNAL = {Ann. of Math. (2)},
  FJOURNAL = {Annals of Mathematics. Second Series},
    VOLUME = {190},
      YEAR = {2019},
    NUMBER = {1},
     PAGES = {249--305},
      ISSN = {0003-486X,1939-8980},
   MRCLASS = {35Q53 (35B30 37K10)},
  MRNUMBER = {3990604},
MRREVIEWER = {John\ Albert},
       DOI = {10.4007/annals.2019.190.1.4},
       URL = {https://doi.org/10.4007/annals.2019.190.1.4},
}

@article {MolinetPilod2012,
    AUTHOR = {Molinet, Luc and Pilod, Didier},
     TITLE = {The {C}auchy problem for the {B}enjamin-{O}no equation in
              {$L^2$} revisited},
   JOURNAL = {Anal. PDE},
  FJOURNAL = {Analysis \& PDE},
    VOLUME = {5},
      YEAR = {2012},
    NUMBER = {2},
     PAGES = {365--395},
      ISSN = {2157-5045,1948-206X},
   MRCLASS = {35Q53 (35B30)},
  MRNUMBER = {2970711},
MRREVIEWER = {John\ Albert},
       DOI = {10.2140/apde.2012.5.365},
       URL = {https://doi.org/10.2140/apde.2012.5.365},
}

@article {KenigPonceVega2001,
    AUTHOR = {Kenig, Carlos E. and Ponce, Gustavo and Vega, Luis},
     TITLE = {On the ill-posedness of some canonical dispersive equations},
   JOURNAL = {Duke Math. J.},
  FJOURNAL = {Duke Mathematical Journal},
    VOLUME = {106},
      YEAR = {2001},
    NUMBER = {3},
     PAGES = {617--633},
      ISSN = {0012-7094,1547-7398},
   MRCLASS = {35R25 (35B30 35B65 35Q51 35Q53)},
  MRNUMBER = {1813239},
MRREVIEWER = {Kenji\ Nakanishi},
       DOI = {10.1215/S0012-7094-01-10638-8},
       URL = {https://doi.org/10.1215/S0012-7094-01-10638-8},
}

@article {CollianderKeelStaffilaniTakaokaTao2003,
    AUTHOR = {Colliander, J. and Keel, M. and Staffilani, G. and Takaoka, H.
              and Tao, T.},
     TITLE = {Sharp global well-posedness for {K}d{V} and modified {K}d{V}
              on {$\mathbb{R}$} and {$\mathbb{T}$}},
   JOURNAL = {J. Amer. Math. Soc.},
  FJOURNAL = {Journal of the American Mathematical Society},
    VOLUME = {16},
      YEAR = {2003},
    NUMBER = {3},
     PAGES = {705--749},
      ISSN = {0894-0347,1088-6834},
   MRCLASS = {35Q53 (35B30 37K10)},
  MRNUMBER = {1969209},
MRREVIEWER = {Vladislav\ G.\ Dubrovsky},
       DOI = {10.1090/S0894-0347-03-00421-1},
       URL = {https://doi.org/10.1090/S0894-0347-03-00421-1},
}

@article {HarropGriffithsKillipVisan2024,
    AUTHOR = {Harrop-Griffiths, Benjamin and Killip, Rowan and Vi\c{s}an,
              Monica},
     TITLE = {Sharp well-posedness for the cubic {NLS} and m{K}d{V} in
              {$H^s(\mathbb{R})$}},
   JOURNAL = {Forum Math. Pi},
  FJOURNAL = {Forum of Mathematics. Pi},
    VOLUME = {12},
      YEAR = {2024},
     PAGES = {Paper No. e6, 86},
      ISSN = {2050-5086},
   MRCLASS = {35Q55 (35Q53 37K10)},
  MRNUMBER = {4726498},
MRREVIEWER = {Feng\ Wang},
       DOI = {10.1017/fmp.2024.4},
       URL = {https://doi.org/10.1017/fmp.2024.4},
}

@article {Guo2011mBO,
    AUTHOR = {Guo, Zihua},
     TITLE = {Local well-posedness and a priori bounds for the modified
              {B}enjamin-{O}no equation},
   JOURNAL = {Adv. Differential Equations},
  FJOURNAL = {Advances in Differential Equations},
    VOLUME = {16},
      YEAR = {2011},
    NUMBER = {11-12},
     PAGES = {1087--1137},
      ISSN = {1079-9389},
   MRCLASS = {35Q53 (35B30 35B45 76B15)},
  MRNUMBER = {2858525},
}

@book {MR4400881,
    AUTHOR = {Klein, Christian and Saut, Jean-Claude},
     TITLE = {Nonlinear dispersive equations---inverse scattering and {PDE}
              methods},
    SERIES = {Applied Mathematical Sciences},
    VOLUME = {209},
 PUBLISHER = {Springer, Cham},
      YEAR = {[2021] \copyright 2021},
     PAGES = {xx+580},
      ISBN = {978-3-030-91426-4; 978-3-030-91427-1},
   MRCLASS = {35-02 (35C08 35Q51 35Q53 35Q55 37K40)},
  MRNUMBER = {4400881},
       DOI = {10.1007/978-3-030-91427-1},
       URL = {https://doi.org/10.1007/978-3-030-91427-1},
}

@article {Nakamura1979,
    AUTHOR = {Nakamura, Akira},
     TITLE = {{$N$}-periodic wave and {$N$}-soliton solutions of the
              modified {B}enjamin-{O}no equation},
   JOURNAL = {J. Phys. Soc. Japan},
  FJOURNAL = {Journal of the Physical Society of Japan},
    VOLUME = {47},
      YEAR = {1979},
    NUMBER = {6},
     PAGES = {2045--2046},
      ISSN = {0031-9015,1347-4073},
   MRCLASS = {76B25},
  MRNUMBER = {559020},
MRREVIEWER = {Ryogo\ Hirota},
       DOI = {10.1143/JPSJ.47.2045},
       URL = {https://doi.org/10.1143/JPSJ.47.2045},
}

@book {LinaresPonce2015,
    AUTHOR = {Linares, Felipe and Ponce, Gustavo},
     TITLE = {Introduction to nonlinear dispersive equations},
    SERIES = {Universitext},
   EDITION = {Second},
 PUBLISHER = {Springer, New York},
      YEAR = {2015},
     PAGES = {xiv+301},
      ISBN = {978-1-4939-2180-5; 978-1-4939-2181-2},
   MRCLASS = {35-01 (35B30 35Q53 35Q55 47F05 47G30)},
  MRNUMBER = {3308874},
       DOI = {10.1007/978-1-4939-2181-2},
       URL = {https://doi.org/10.1007/978-1-4939-2181-2},
}

@book {MR2233925,
    AUTHOR = {Tao, Terence},
     TITLE = {Nonlinear dispersive equations},
    SERIES = {CBMS Regional Conference Series in Mathematics},
    VOLUME = {106},
      NOTE = {Local and global analysis},
 PUBLISHER = {Published for the Conference Board of the Mathematical
              Sciences, Washington, DC; by the American Mathematical
              Society, Providence, RI},
      YEAR = {2006},
     PAGES = {xvi+373},
      ISBN = {0-8218-4143-2},
   MRCLASS = {35Q53 (35B35 35P25 35Q55 37K10)},
  MRNUMBER = {2233925},
MRREVIEWER = {Sebastian\ Herr},
       DOI = {10.1090/cbms/106},
       URL = {https://doi.org/10.1090/cbms/106},
}

@article {Miura1976,
    AUTHOR = {Miura, Robert M.},
     TITLE = {The {K}orteweg-de {V}ries equation: a survey of results},
   JOURNAL = {SIAM Rev.},
  FJOURNAL = {SIAM Review. A Publication of the Society for Industrial and
              Applied Mathematics},
    VOLUME = {18},
      YEAR = {1976},
    NUMBER = {3},
     PAGES = {412--459},
      ISSN = {1095-7200},
   MRCLASS = {35Q99},
  MRNUMBER = {404890},
MRREVIEWER = {Ronald\ Di\ Perna},
       DOI = {10.1137/1018076},
       URL = {https://doi.org/10.1137/1018076},
}

@article {MiuraGardnerKruskal1968,
    AUTHOR = {Miura, Robert M. and Gardner, Clifford S. and Kruskal, Martin
              D.},
     TITLE = {Korteweg-de {V}ries equation and generalizations. {II}.
              {E}xistence of conservation laws and constants of motion},
   JOURNAL = {J. Mathematical Phys.},
  FJOURNAL = {Journal of Mathematical Physics},
    VOLUME = {9},
      YEAR = {1968},
     PAGES = {1204--1209},
      ISSN = {0022-2488,1089-7658},
   MRCLASS = {35.36 (82.00)},
  MRNUMBER = {252826},
MRREVIEWER = {Joel\ Smoller},
       DOI = {10.1063/1.1664701},
       URL = {https://doi.org/10.1063/1.1664701},
}

@article {MR4902804,
    AUTHOR = {Pineau, Ben and Taylor, Mitchell A.},
     TITLE = {Global well-posedness for the generalized derivative nonlinear
              {S}chr\"{o}dinger equation},
   JOURNAL = {Nonlinearity},
  FJOURNAL = {Nonlinearity},
    VOLUME = {38},
      YEAR = {2025},
    NUMBER = {5},
     PAGES = {Paper No. 055022, 67},
      ISSN = {0951-7715,1361-6544},
   MRCLASS = {35Q55 (35A01 35A02)},
  MRNUMBER = {4902804},
MRREVIEWER = {Dongfeng\ Yan},
       DOI = {10.1088/1361-6544/add076},
       URL = {https://doi.org/10.1088/1361-6544/add076},
}

@article {IfrimTataru2024,
    AUTHOR = {Ifrim, Mihaela and Tataru, Daniel},
     TITLE = {The compressible {E}uler equations in a physical vacuum: a
              comprehensive {E}ulerian approach},
   JOURNAL = {Ann. Inst. H. Poincar\'{e} C Anal. Non Lin\'{e}aire},
  FJOURNAL = {Annales de l'Institut Henri Poincar\'{e} C. Analyse Non
              Lin\'{e}aire},
    VOLUME = {41},
      YEAR = {2024},
    NUMBER = {2},
     PAGES = {405--495},
      ISSN = {0294-1449,1873-1430},
   MRCLASS = {35Q75 (35L10 35Q31 35Q35 35R35)},
  MRNUMBER = {4721644},
       DOI = {10.4171/aihpc/91},
       URL = {https://doi.org/10.4171/aihpc/91},
}

@article {IfrimPineauTataruTaylor2025,
    AUTHOR = {Ifrim, Mihaela and Pineau, Ben and Tataru, Daniel and Taylor,
              Mitchell A.},
     TITLE = {Sharp {H}adamard local well-posedness, enhanced uniqueness and
              pointwise continuation criterion for the incompressible free
              boundary {E}uler equations},
   JOURNAL = {Ann. PDE},
  FJOURNAL = {Annals of PDE. Journal Dedicated to the Analysis of Problems
              from Physical Sciences},
    VOLUME = {11},
      YEAR = {2025},
    NUMBER = {1},
     PAGES = {Paper No. 16, 154},
      ISSN = {2524-5317,2199-2576},
   MRCLASS = {35B60 (35Q31 76B03 76B15)},
  MRNUMBER = {4915720},
MRREVIEWER = {Yongfu\ Wang},
       DOI = {10.1007/s40818-025-00204-4},
       URL = {https://doi.org/10.1007/s40818-025-00204-4},
}

@article {BockKruskal1979,
    AUTHOR = {Bock, T. L. and Kruskal, M. D.},
     TITLE = {A two-parameter {M}iura transformation of the {B}enjamin-{O}no
              equation},
   JOURNAL = {Phys. Lett. A},
  FJOURNAL = {Physics Letters. A},
    VOLUME = {74},
      YEAR = {1979},
    NUMBER = {3-4},
     PAGES = {173--176},
      ISSN = {0375-9601,1873-2429},
   MRCLASS = {35Q20 (58F07 76B15)},
  MRNUMBER = {591320},
       DOI = {10.1016/0375-9601(79)90762-X},
       URL = {https://doi.org/10.1016/0375-9601(79)90762-X},
}

@article {MR2383514,
    AUTHOR = {Dawson, L. and McGahagan, H. and Ponce, G.},
     TITLE = {On the decay properties of solutions to a class of
              {S}chr\"{o}dinger equations},
   JOURNAL = {Proc. Amer. Math. Soc.},
  FJOURNAL = {Proceedings of the American Mathematical Society},
    VOLUME = {136},
      YEAR = {2008},
    NUMBER = {6},
     PAGES = {2081--2090},
      ISSN = {0002-9939,1088-6826},
   MRCLASS = {35Q55 (35B40)},
  MRNUMBER = {2383514},
       DOI = {10.1090/S0002-9939-08-09355-6},
       URL = {https://doi.org/10.1090/S0002-9939-08-09355-6},
}

@article {Saut79,
    AUTHOR = {Saut, J.-C.},
     TITLE = {Sur quelques g\'{e}n\'{e}ralisations de l'\'{e}quation de
              {K}orteweg-de {V}ries},
   JOURNAL = {J. Math. Pures Appl. (9)},
  FJOURNAL = {Journal de Math\'{e}matiques Pures et Appliqu\'{e}es.
              Neuvi\`eme S\'{e}rie},
    VOLUME = {58},
      YEAR = {1979},
    NUMBER = {1},
     PAGES = {21--61},
      ISSN = {0021-7824,1776-3371},
   MRCLASS = {35Q20 (76B99)},
  MRNUMBER = {533234},
MRREVIEWER = {Amy\ Cohen},
}

@article {MR1044731,
    AUTHOR = {Abdelouhab, L. and Bona, J. L. and Felland, M. and Saut,
              J.-C.},
     TITLE = {Nonlocal models for nonlinear, dispersive waves},
   JOURNAL = {Phys. D},
  FJOURNAL = {Physica D. Nonlinear Phenomena},
    VOLUME = {40},
      YEAR = {1989},
    NUMBER = {3},
     PAGES = {360--392},
      ISSN = {0167-2789,1872-8022},
   MRCLASS = {58D25 (35Q53)},
  MRNUMBER = {1044731},
MRREVIEWER = {Peter\ L.\ Christiansen},
       DOI = {10.1016/0167-2789(89)90050-X},
       URL = {https://doi.org/10.1016/0167-2789(89)90050-X},
}

@article {FokasAblowitz1983,
    AUTHOR = {Fokas, A. S. and Ablowitz, M. J.},
     TITLE = {The inverse scattering transform for the {B}enjamin-{O}no
              equation---a pivot to multidimensional problems},
   JOURNAL = {Stud. Appl. Math.},
  FJOURNAL = {Studies in Applied Mathematics},
    VOLUME = {68},
      YEAR = {1983},
    NUMBER = {1},
     PAGES = {1--10},
      ISSN = {0022-2526},
   MRCLASS = {35Q20 (44A15 76V05)},
  MRNUMBER = {686248},
MRREVIEWER = {Amy\ Cohen},
       DOI = {10.1002/sapm19836811},
       URL = {https://doi.org/10.1002/sapm19836811},
}

@article {Lax1968,
    AUTHOR = {Lax, Peter D.},
     TITLE = {Integrals of nonlinear equations of evolution and solitary
              waves},
   JOURNAL = {Comm. Pure Appl. Math.},
  FJOURNAL = {Communications on Pure and Applied Mathematics},
    VOLUME = {21},
      YEAR = {1968},
     PAGES = {467--490},
      ISSN = {0010-3640,1097-0312},
   MRCLASS = {35.79 (47.00)},
  MRNUMBER = {235310},
MRREVIEWER = {Gerald\ Rosen},
       DOI = {10.1002/cpa.3160210503},
       URL = {https://doi.org/10.1002/cpa.3160210503},
}

@article {FokasFuchssteiner1981,
    AUTHOR = {Fokas, A. S. and Fuchssteiner, B.},
     TITLE = {The hierarchy of the {B}enjamin-{O}no equation},
   JOURNAL = {Phys. Lett. A},
  FJOURNAL = {Physics Letters. A},
    VOLUME = {86},
      YEAR = {1981},
    NUMBER = {6-7},
     PAGES = {341--345},
      ISSN = {0375-9601,1873-2429},
   MRCLASS = {58F07 (35Q20 76V05)},
  MRNUMBER = {637024},
       DOI = {10.1016/0375-9601(81)90551-X},
       URL = {https://doi.org/10.1016/0375-9601(81)90551-X},
}

@article {KaupMatsuno1998,
    AUTHOR = {Kaup, D. J. and Matsuno, Y.},
     TITLE = {The inverse scattering transform for the {B}enjamin-{O}no
              equation},
   JOURNAL = {Stud. Appl. Math.},
  FJOURNAL = {Studies in Applied Mathematics},
    VOLUME = {101},
      YEAR = {1998},
    NUMBER = {1},
     PAGES = {73--98},
      ISSN = {0022-2526,1467-9590},
   MRCLASS = {34L25 (34A55 35Q53 37K15 76B99)},
  MRNUMBER = {1627325},
MRREVIEWER = {H.\ J. S. Dorren},
       DOI = {10.1111/1467-9590.00086},
       URL = {https://doi.org/10.1111/1467-9590.00086},
}

\end{document}